\documentclass[11pt, final, reqno]{amsart}

\usepackage[utf8]{inputenc}
\usepackage[T1]{fontenc}
\usepackage[english]{babel}
\usepackage{amsmath, amsthm}
\usepackage{ae}
\usepackage{icomma}
\usepackage{units}
\usepackage{color}
\usepackage{graphicx}
\usepackage{bbm}
\usepackage{caption}
\usepackage{array}
\usepackage[hmarginratio=1:1]{geometry}
\usepackage[hyphens]{url}
\usepackage[pdfpagelabels=false, hidelinks]{hyperref}
\usepackage{mathrsfs}
\usepackage{amssymb}
\usepackage{placeins} 
\usepackage{tikz}
\usepackage{float} 
\usepackage[nodayofweek]{datetime}
\usepackage{enumitem}
\usepackage{mathtools}
\usepackage[numbers, sort]{natbib}
\usepackage{dsfont}

\mathtoolsset{showonlyrefs=true}

\newcommand{\R}{\ensuremath{\mathbb{R}}}
\newcommand{\N}{\ensuremath{\mathbb{N}}}
\newcommand{\Q}{\ensuremath{\mathbb{Q}}}
\newcommand{\Z}{\ensuremath{\mathbb{Z}}}

\newcommand{\C}{\ensuremath{\mathbb{C}}}
\DeclareMathOperator{\dist}{\textnormal{dist}}
\DeclareMathOperator{\Tr}{Tr}
\newcommand{\eps}{\ensuremath{\varepsilon}}
\renewcommand{\P}{\ensuremath{\mathcal{P}}}
\DeclareMathOperator{\Res}{Res}
\DeclareMathOperator{\osc}{Osc}
\DeclareMathOperator{\argmax}{argmax}
\newcommand{\lf}{\ensuremath{\lfloor}}
\newcommand{\rf}{\ensuremath{\rfloor}}

\newtheorem{theorem}{Theorem}[section]

\newtheorem{lemma}[theorem]{Lemma}

\newtheorem{corollary}[theorem]{Corollary}
\newtheorem{conjecture}[theorem]{Conjecture}
\numberwithin{theorem}{section}
\numberwithin{definition}{section}
\theoremstyle{remark}
\newtheorem{remark}[theorem]{Remark}

\newcommand{\limplus}{{\mathchoice{\vcenter{\hbox{$\scriptstyle +$}}}
  {\vcenter{\hbox{$\scriptstyle +$}}}
  {\vcenter{\hbox{$\scriptscriptstyle +$}}}
  {\vcenter{\hbox{$\scriptscriptstyle +$}}}
}}
\newcommand{\limminus}{{\mathchoice{\vcenter{\hbox{$\scriptstyle -$}}}
  {\vcenter{\hbox{$\scriptstyle -$}}}
  {\vcenter{\hbox{$\scriptscriptstyle -$}}}
  {\vcenter{\hbox{$\scriptscriptstyle -$}}}
}}
\newcommand{\limpm}{{\mathchoice{\vcenter{\hbox{$\scriptstyle \pm$}}}
  {\vcenter{\hbox{$\scriptstyle \pm$}}}
  {\vcenter{\hbox{$\scriptscriptstyle \pm$}}}
  {\vcenter{\hbox{$\scriptscriptstyle \pm$}}}
}}

\begin{document}

\title[Maximizing Riesz means of anisotropic harmonic oscillators]{Maximizing Riesz means of anisotropic harmonic oscillators}

\author[S. Larson]{Simon Larson}
\address{\textnormal{(S. Larson)} Department of Mathematics, KTH Royal Institute of Technology, SE-100 44 Stockholm, Sweden}
\email{simla@math.kth.se}


\begin{abstract}
  We consider problems related to the asymptotic minimization of eigenvalues of anisotropic harmonic oscillators in the plane. In particular we study Riesz means of the eigenvalues and the trace of the corresponding heat kernels. The eigenvalue minimization problem can be reformulated as a lattice point problem where one wishes to maximize the number of points of $(\mathbb{N}-\tfrac12)\times(\mathbb{N}-\tfrac12)$ inside triangles with vertices $(0, 0), (0, \lambda \sqrt{\beta})$ and $(\lambda/{\sqrt{\beta}}, 0)$ with respect to $\beta>0$, for fixed $\lambda\geq 0$. This lattice point formulation of the problem naturally leads to a family of generalized problems where one instead considers the shifted lattice $(\mathbb{N}+\sigma)\times(\mathbb{N}+\tau)$, for $\sigma, \tau >-1$. We show that the nature of these problems are rather different depending on the shift parameters, and in particular that the problem corresponding to harmonic oscillators, $\sigma=\tau=-\tfrac12$, is a critical case.
\end{abstract}

\maketitle


\section{Introduction and main result}
\label{sec:Intro}
For $\beta>0$, let $L_\beta$ denote the self-adjoint operator on $L^2(\R^2)$ acting as
\begin{equation}
  -\Delta+ \beta x^2+ \beta^{-1}y^2, 
\end{equation}
which we will refer to as the \emph{anisotropic harmonic oscillator}.
For any $\beta>0$ the spectrum of $L_\beta$ is positive and purely discrete, consisting of an infinite number of eigenvalues. Let $\{\lambda_k(\beta)\}_{k\in \N}$ denote the eigenvalues of $L_\beta$ numbered in increasing order and each repeated according to its multiplicity. Here and in what follows we use the convention that $\N=\{1, 2, \ldots\}$. It is well known that the eigenvalues have a one-to-one correspondence with $\N^2$, explicitly given by
\begin{equation}\label{eq:eigenvalues_osc}
    (k_1, k_2)\mapsto 2(k_1-1/2)\sqrt{\beta}+2(k_2-1/2)/{\sqrt{\beta}}=\lambda_{(k_1, k_2)}(\beta).
\end{equation}

In this paper we consider a number of problems related to the following question: Given $k\in \N$ for what values of $\beta$ is the minimum
\begin{equation}
  \min\{\lambda_k(\beta): \beta >0\}
\end{equation}
realized? In particular we are interested in how the set of minimizing $\beta$ behaves as $k$ tends to infinity.
Similar questions concerning minimizing or maximizing functions of the spectrum of differential operators has in recent years seen large interest, see for instance~\cite{MR3681143} and references therein.

\subsection{Minimizing eigenvalues and counting lattice points}
The problem of minimizing the $k$-th eigenvalue among the operators $L_\beta$ can be reformulated as finding the $\beta$ for which the \emph{eigenvalue counting function}, 
\begin{equation}\label{eq:eigenvalue_countingfunc}
  N(\beta, \lambda):= \#\{j\in \N: \lambda_j(\beta)\leq \lambda\}, 
\end{equation}
is first to reach $k$. Hence, if one understands the maximization problem
\begin{equation}\label{eq:maximization_prob}
  \max\{N(\beta, \lambda): \beta>0 \}
\end{equation}
for all $\lambda\geq 0$, then one also understands the problem of minimizing $\lambda_k(\beta)$ for any $k\in \N$. 

Due to the form of the eigenvalues of $L_\beta$ this maximization problem can be reformulated as a geometric lattice point problem: 
Given $\lambda \geq 0$ find the triangle, amongst those given by the vertices $(0, 0), (\lambda/{\sqrt{\beta}}, 0)$ and $(0, \sqrt{\beta}\lambda)$, which contains the greatest number of points of the lattice $(\N-\tfrac12)\times(\N - \tfrac12)$. (We have here rescaled the problem to avoid the factor $2$ appearing in the explicit form of the eigenvalues~\eqref{eq:eigenvalues_osc}.)

In a similar manner the problem of minimizing eigenvalues of the Dirichlet Laplacian among cuboids of unit measure, i.e.\ domains of the form $Q=(0, a_1)\times\ldots \times (0, a_d)\subset \R^d$ with $\prod_{i=1}^d a_i =1$, can be recast as finding which ellipsoid centered at the origin and of fixed volume contains the largest number of positive integer lattice points. In~\cite{MR3001382} Antunes and Freitas used this idea to show that if $Q_k$ is a sequence of unit area rectangles such that $Q_k$ minimizes $\lambda_k$ then $Q_k$ converges to the square as $k$ tends to infinity. In~\cite{MR3556370} a similar result was proven for the case of the Neumann Laplacian. The result of Antunes and Freitas was generalized to the three-dimensional case in~\cite{vdBergGittins}, and to arbitrary dimension in~\cite{GittinsLarson} where also the corresponding Neumann result was proven to hold in any dimension. 

Generalizing the work of Antunes and Freitas from the viewpoint of lattice point problems, Laugesen and Liu~\cite{LaugesenLiu1} studied the following problem: Let $f\colon [0, \infty)\to \R$ be a strictly decreasing concave function with $f(0)=1$ and $f(1)=0$. Define, for $s, r>0$, the function  
\begin{equation}\label{eq:Counter_under_graph}
  N(s, r):=\#\{(k_1, k_2)\in \N^2: k_2\leq r s f(k_1 s/r)\}.
\end{equation}
This function counts the number of integer lattice points under the graph of $f$ after it has been compressed in the $x$-direction by a factor $s$, stretched in the $y$-direction by the same factor, and scaled by a factor $r$.
What happens to the set of maximizers, $\argmax_{s>0} N(s, r)$, as $r$ (the area under the rescaled graph) tends to infinity? 
For a large family of functions $f$ they prove that the maximizing set of $s$ tends to~$1$. 
The corresponding problem with concave curves replaced by convex ones was treated in~\cite{AriturkLaugesen}. More recently Laugesen and Liu~\cite{LaugesenLiu2} have studied the case of both concave and convex curves where they also allow for shifting the lattice, i.e.\ replacing $\N^2$ by $(\N+\sigma)\times (\N+\tau)$. 
For work on similar problems in higher dimensions see also~\cite{Marshall, GuoWang}.

However, the results of~\cite{AriturkLaugesen, LaugesenLiu1, LaugesenLiu2} all require that the graph of the function $f$ has non-vanishing curvature. In particular, the case of $f(x)=1-x$ is not covered, which is precisely the problem of interest here. 
That the case of vanishing curvature is excluded from the results of~\cite{LaugesenLiu1, AriturkLaugesen, LaugesenLiu2, Marshall} is no accident, and also more classical problems in lattice point theory are less well understood in this setting~\cite{MR2310176, MR3309484}. In fact it was conjectured in~\cite{LaugesenLiu1} that the problem with $f(x)=1-x$ fails to have an asymptotic maximizer (see also~\cite{LaugesenLiu2} for the shifted case), and that instead the sequence of maximizing values of $s$ has an infinite number of limit points. In~\cite{MarshallSteinerberger} Marshall and Steinerberger prove the conjecture in the case of the non-shifted lattice $\N^2$.

\subsection{Maximizing Riesz means}
In what follows we will consider a family of problems closely related to the maximization problem in~\eqref{eq:maximization_prob}. The main problem that we are interested in is the behavior of $\beta$ which maximizes the function
\begin{equation}\label{eq:def_Reisz_mean}
  R^\gamma_{\sigma, \tau}(\beta, \lambda):= \sum_{k\in\N^2}(\lambda-(k_1+\sigma)\sqrt{\beta}-(k_2+\tau)/{\sqrt{\beta}})_\limplus^\gamma, 
\end{equation}
for $\gamma>0$ and $\sigma, \tau>-1$, as $\lambda$ tends to infinity (if $\sigma=\tau$ we will write simply $R^\gamma_\sigma$). Here and in what follows $x_\limpm:=(|x|\pm x)/2.$

Setting $\gamma=0$ and interpreting the sum appropriately,~\eqref{eq:def_Reisz_mean} reduces to the function
\begin{equation}\label{eq:count_as_Rmean}
  N_{\sigma, \tau}(\beta, \lambda):=\#\{(k_1, k_2)\in (\N+\sigma)\times (\N+\tau) : k_1 \sqrt{\beta}+k_2/{\sqrt{\beta}} \leq \lambda \}.
\end{equation}
If $\sigma=\tau=0$ then~\eqref{eq:count_as_Rmean} corresponds to the case considered in~\cite{LaugesenLiu1, MarshallSteinerberger}.
If $\sigma=\tau=-1/2$ then~\eqref{eq:count_as_Rmean} is the eigenvalue counting function~\eqref{eq:eigenvalue_countingfunc} evaluated at $2\lambda$. Similarly, for $\gamma>0$, $R^\gamma_{\limminus1/2}(\beta, \lambda)= \Tr(L_\beta-2\lambda)_\limminus^\gamma$ is the \emph{Riesz mean} of order $\gamma$ of $L_\beta$. Here we will adopt this name also for other $\sigma$ and $\tau$.

Taking $\gamma>0$ (instead of $\gamma=0$ as in the original problem) leads to a regularization of the problem and will allow us to use certain tools that are effectively excluded in the case of the counting function. Using the Aizenman--Lieb Identity~\cite{AizenmanLieb} the regularizing effect of increasing $\gamma$ becomes clear as it allows one to write $R^\gamma_{\sigma, \tau}$ as a weighted mean of lower order Riesz means: for $\gamma_2 >\gamma_1\geq 0$ and $\lambda \geq 0$,
\begin{equation}\label{eq:Aizenman_Lieb}
  R^{\gamma_2}_{\sigma, \tau}(\lambda) = B(1+\gamma_1, \gamma_2-\gamma_1)^{-1}\int_0^\infty \eta^{-1+\gamma_2-\gamma_1}R^{\gamma_1}_{\sigma, \tau}(\lambda-\eta)\, d\eta, 
\end{equation}
where $B$ denotes the Euler Beta function, and we as above interpret $R^0_{\sigma, \tau}$ as $N_{\sigma, \tau}$. This identity follows from linearity and the fact that
\begin{equation}
  \int_0^\infty \tau^{-1+\gamma_2-\gamma_2}(\tau-a)_\limplus^{\gamma_1}\, d\tau = a_\limplus^{\gamma_2} B(1+\gamma_1, \gamma_2-\gamma_1).
\end{equation}

We will also consider a further regularized problem which in the harmonic oscillator case corresponds to the \emph{trace of the heat kernel} of $L_\beta$, that is $\Tr(e^{-t L_\beta})$. For general shift parameters $\sigma, \tau$ we define
\begin{equation}\label{eq:Heat_trace_intro}
  H_{\sigma, \tau}(\beta, t) := \sum_{k\in \N^2} e^{-t((k_1+\sigma)\sqrt{\beta}+(k_2+\tau)/{\sqrt{\beta}})}.
\end{equation}
The problem of asymptotically maximizing this function in $\beta$ as $t\to 0^\limplus$ can in a certain sense be seen as a limiting version of the Riesz mean problems with $\lambda$ and $\gamma$ going to infinity simultaneously. A further connection to the Riesz means can be found by noticing that $H_{\sigma, \tau}$ can be written using the Laplace transform of the $R^\gamma_{\sigma, \tau}$:
\begin{equation}\label{eq:heat_trace_as_Laplace_transform}
  H_{\sigma, \tau}(\beta, t) = \frac{t^{1+\gamma}}{\Gamma(1+\gamma)}\int_0^\infty R^\gamma_{\sigma, \tau}(\beta, \lambda)e^{-\lambda t}\, d\lambda.
\end{equation}
This connection via the Laplace transform of $R^\gamma_{\sigma, \tau}$ and $H_{\sigma, \tau}$, combined with the fact that $H_{\sigma, \beta}$ can be explicitly computed, will be of importance when we study the behavior of the Riesz means for large $\lambda$ (following~\cite{MR1062904, MR1079775}). The main motivation for including the study of the heat kernel problem here is that it is easier to understand than the Riesz mean problem, and can thus serve as a guide to what we might expect when studying $R^\gamma_{\sigma, \tau}$.

\subsection{Main results and conjectures}
\label{sec:Main_results}

Throughout the paper $\beta^\gamma_{\sigma, \tau}(\lambda)$, for $\lambda\geq0, $ will denote a $\beta$ which maximizes $R^\gamma_{\sigma, \tau}(\, \cdot\, , \lambda)$, that is, satisfies
\begin{equation}
  R^\gamma_{\sigma, \tau}(\beta^\gamma_{\sigma, \tau}(\lambda), \lambda) = \max\{ R^\gamma_{\sigma, \tau}(\beta, \lambda): \beta >0\}.
\end{equation}
 As such a maximizer is not necessarily unique we emphasize that when we make a claim concerning $\beta^\gamma_{\sigma, \tau}(\lambda)$ we mean that this holds for \emph{all} maximizers. Similarly we let $\beta_{\sigma, \tau}^H(t)$, with $t>0$, denote a maximizer of $H_{\sigma, \tau}(\, \cdot\, , t)$, i.e.\ such that
 \begin{equation}
   \label{eq:MaximizerHeatTrace}
  H_{\sigma, \tau}(\beta^H_{\sigma, \tau}(t), t) = \max\{ H_{\sigma, \tau}(\beta, t): \beta >0\}.
 \end{equation}

We first turn to what we are able to prove for $\beta^H_{\sigma, \tau}(t)$.
The problem is made easier due to the fact that the sum~\eqref{eq:Heat_trace_intro} can be explicitly computed:
\begin{equation}\label{eq:Htrace_explicit_form}
  H_{\sigma, \tau}(\beta, t)= \sum_{k \in \N^2} e^{-t((k_1+\sigma)\sqrt{\beta}+(k_2+\tau)/{\sqrt{\beta}})} = \frac{e^{-t (\sigma\sqrt{\beta} +\tau/{\sqrt{\beta}})}}{(e^{t\sqrt{\beta}}-1)(e^{t/{\sqrt{\beta}}}-1)}.
\end{equation}
The question of maximizing with respect to $\beta$ is thus reduced to an explicit optimization problem in one variable. However, the behavior of this function depends strongly on the parameters $t, \sigma, \tau$ and carrying out the maximization explicitly is difficult.  

For $\beta_{\sigma, \tau}^H(t)$ there are two asymptotic regions that we wish to study: when $t\to 0^\limplus$ and when $t\to\infty$. The asymptotic problem $t\to 0^\limplus$ is most closely related to that studied for the Riesz means as more and more of the lattice points (eigenvalues) become relevant as $t$ becomes smaller, while if $t$ goes to $\infty$ the main contribution comes from the lattice points which are closest to the origin. Our first theorem tells us that we can determine the behavior of $\beta_{\sigma, \tau}^H(t)$ in both limits.

\begin{theorem}\label{thm:HeatKernel}
  For each $t>0$ and $\sigma, \tau>-1$ there exists a maximizing value $\beta^H_{\sigma, \tau}(t)$ satisfying~\eqref{eq:MaximizerHeatTrace}. If\/ $\max\{\sigma, \tau\}\geq -1/2$ then the maximizer is unique for each $t>0$, moreover, if $\sigma=\tau\geq -1/2$ then $\beta^H_{\sigma}(t)=1$.
  
  Furthermore, for all $\sigma, \tau > -1$, it holds that
  \begin{equation}
    \lim_{t\to \infty} \beta^H_{\sigma, \tau}(t) = \frac{1+\tau}{1+\sigma}, 
  \end{equation}
  similarly, for all $\sigma, \tau >-1/2$, 
  \begin{equation}
  \lim_{t\to 0^\limplus} \beta^H_{\sigma, \tau}(t)=\frac{1+2\tau}{1+2\sigma}.
  \end{equation}
  For all values of $\sigma, \tau>-1$ not covered above, any sequence of maximizers degenerates, i.e.~$\beta^H_{\sigma, \tau}(t)$ tends to $0$ or $\infty$ as $t\to 0^\limplus$.
\end{theorem}

\begin{remark}  
  One should note that the asymptotic maximizer in the limit $t\to \infty$ is precisely the $\beta$ which minimizes the area of the first triangle containing any lattice points at all. In the limit $t\to 0^\limplus$ we find the same limit as Laugesen--Liu~\cite{LaugesenLiu2} found for the counting function~\eqref{eq:Counter_under_graph}. This limit corresponds to balancing the area of the region below the bounding curve (in our case a line) to the left of the first column of lattice points, with that of the region below the bounding curve and below the first row of lattice points (see~\cite[Figure~1]{LaugesenLiu2}).
\end{remark}

In the same direction we prove the following for Riesz means: 
\begin{theorem}\label{thm:Mainconvtheorem}
  For all $\gamma>0$ and $\sigma, \tau >-1/2$ it holds that
  \begin{equation}
    \lim_{\lambda \to \infty} \beta^\gamma_{\sigma, \tau}(\lambda) = \frac{1+2\tau}{1+2\sigma}.
  \end{equation}
\end{theorem}

That is, for all shifts $\sigma, \tau>-1/2$ any sequence of maximizers, with $\lambda \to \infty$, for positive order Riesz means has a unique limit. Thus the behavior observed in~\cite{LaugesenLiu1} and studied in~\cite{MarshallSteinerberger} for the counting function with $\sigma=\tau=0$ effectively vanishes as soon as we consider the regularized problem of Riesz means with $\gamma>0$. 

In the case of the harmonic oscillators, $\sigma=\tau=-1/2$, we find a unique limit first when $\gamma>1$. Specifically we prove that:
\begin{theorem}\label{thm:Mainconvtheorem_oscillators}
  For all $\gamma>1$ it holds that
  \begin{equation}
    \lim_{\lambda \to \infty} \beta^\gamma_{\limminus1/2}(\lambda) = 1.
  \end{equation}
\end{theorem}

We do not believe that the failure to prove the corresponding result for smaller $\gamma$ is a result of our methods, but that in these cases the behavior of the maximizers resembles that in~\cite{MarshallSteinerberger}. In fact, for the cases that are not covered by the above we conjecture the following, which extends the conjecture of Laugesen and Liu~\cite{LaugesenLiu1}:

\begin{conjecture}\label{conj:conjectured_limits}
  The conjecture is split into two parts: 
  \begin{enumerate}[label=(\roman*)]
    \item\label{itm:conj_1} For all $\sigma, \tau >-1/2$ the set
    \begin{equation}
      \bigcap_{\lambda >0}\overline{\bigcup_{\lambda'> \lambda}\underset{\beta>0}{\argmax}\; N_{\sigma, \tau}(\beta, \lambda')}
    \end{equation}
    is infinite.

    \item\label{itm:conj_2} For all $0\leq \gamma \leq 1$ the set
    \begin{equation}
      \bigcap_{\lambda >0}\overline{\bigcup_{\lambda'> \lambda}\underset{\beta>0}{\argmax}\; R^\gamma_{\limminus1/2}(\beta, \lambda')}
    \end{equation}
    is infinite.
  \end{enumerate}
\end{conjecture}

As mentioned earlier the case $\gamma=\sigma=\tau=0$ was recently settled by Marshall and Steinerberger~\cite{MarshallSteinerberger}.

\subsection{Idea of proof}
The conjecture, as well as the proof of Theorems~\ref{thm:Mainconvtheorem} and~\ref{thm:Mainconvtheorem_oscillators}, is based on precise asymptotic expansions of $R^\gamma_{\sigma, \tau}(\beta, \lambda)$ as $\lambda \to \infty$. In~\cite{MR1061661, MR1079775} the authors study the asymptotic behavior of $R_{\limminus1/2}(1, \lambda/2)=\Tr((-\Delta+|x|^2)-\lambda)_\limminus^\gamma$ in connection to Lieb--Thirring inequalities (see also~\cite{MR1686426, MR1747896}). The calculations carried out there transfer without much change to what we study here, see Section~\ref{sec:Proof_asymptotics}.

Let $\zeta\colon \C\times \C \to \C$ denote the Hurwitz $\zeta$-function. In the special case $\zeta(z, 1)$ this is the Riemann $\zeta$-function which we denote simply by $\zeta(z)$~\cite[Chapter~25]{NIST}. Let also $\{x\}$ denote the fractional part of $x\in \R$, i.e.\ $\{x\}=x -\lfloor x\rfloor$. 

\begin{theorem}\label{thm:asymptotics}
  For any $\gamma>0$, $M\in \N$, $\delta >0$, $\beta \in \R_\limplus$ and $\sigma, \tau > -1$, there are constants $\alpha_k=\alpha_k(\beta, \sigma, \tau, \gamma)$ such that
  \begin{align}
  R^\gamma_{\sigma, \tau}(\beta, \lambda) 
     &=
     \sum_{k=0}^{M+1} \alpha_{k}\lambda^{2-k+\gamma}
     +
     \osc(\beta, \lambda) + o(\lambda^{-M+\gamma + \delta}), 
  \end{align} 
  as $\lambda \to \infty$. 
  The coefficients $\alpha_k$ are continuous in $\beta$ and\/ $|{\osc(\beta, \lambda)}|\leq C_\beta(\lambda+1)$. Moreover, $C_\beta$ and the implicit constant of the remainder term are uniformly bounded for $\beta$ in compact subsets of\/ $\R_\limplus$. 

\begin{enumerate}[label=\textup{(}\roman*\textup{)}]
  \item\label{thm:Asymptotics_part_I_rational} If $\beta=\tfrac\mu\nu \in \Q_\limplus$, $\gcd(\mu, \nu)=1$, then, with $x=\sqrt{\mu\nu}\lambda-\mu\sigma-\nu\tau$, 
  \begin{align}
    \hspace{12pt}\osc(\beta, \lambda)
    &= 
    \frac{\zeta(-\gamma, \{x\})}{(\mu\nu)^{\frac{1+\gamma}{2}}}\lambda
    - \frac{\zeta(-1-\gamma, \{x\})}{(\mu\nu)^{1+\frac\gamma2}}
     - \frac{(1+2\sigma)\mu+(1+2\tau)\nu}
     {2(\mu\nu)^{1+\frac{\gamma}{2}}}
     \zeta(-\gamma, \{x\})\\
    &\quad+
     \frac{\nu^{\gamma/2}\Gamma(1+\gamma)}{\mu^{\gamma/2}(2\pi)^{1+\gamma}}
     \sum_{\substack{k\in \N\\ k/\nu \notin \N}}
    \frac{\sin(\pi k (2x-\mu)/\nu- \frac{\pi}{2}(1+\gamma))}{k^{1+\gamma}\sin(\pi k \tfrac\mu\nu)} \\
  &\quad +
   \frac{\mu^{\gamma/2}\Gamma(1+\gamma)}{\nu^{\gamma/2}(2\pi)^{1+\gamma}}
   \sum_{\substack{k\in \N\\ k/\mu \notin \N}}
   \frac{\sin(\pi k (2x-\nu)/\mu- \frac{\pi}{2}(1+\gamma))}{k^{1+\gamma}\sin(\pi k\tfrac\nu\mu)}.
  \end{align}

  \item\label{thm:Asymptotics_part_II_irrational}  If $\beta\in \R_\limplus{\setminus}\Q$, it holds that
  \begin{align}
     \hspace{26pt}\osc(\beta, \lambda) 
     &=
     \frac{\beta^{-\gamma/2}\Gamma(1+\gamma)}{(2\pi)^{1+\gamma}}
     \sum_{k=1}^{\Lambda(\lambda)/{\sqrt{\beta}}}
  \frac{\sin(\pi k (2\lambda \sqrt{\beta}-(1+2\sigma) \beta-2\tau)- \frac{\pi}{2}(1+\gamma))}{k^{1+\gamma}\sin(\pi k \beta)} \\
  &\quad +
   \frac{\beta^{\gamma/2}\Gamma(1+\gamma)}{(2\pi)^{1+\gamma}}
   \sum_{k=1}^{\Lambda(\lambda)\sqrt{\beta}}
   \frac{\sin(\pi k (2\lambda/{\sqrt{\beta}}-2\sigma-(1+2\tau)/\beta)- \frac{\pi}{2}(1+\gamma))}{k^{1+\gamma}\sin(\pi k/\beta)} \\
   &\quad + o(\lambda^{-M+\gamma + \delta}), 
  \end{align}
   where $\Lambda(\lambda)=O(\lambda^{\frac{M+2-\gamma}{\gamma}})$.
 \end{enumerate}
\end{theorem}

\begin{remark}
  A couple of remarks are in order:
  \begin{enumerate}
    \item If $\gamma\in \N$ then $\alpha_k=0$ for all $k>2+\gamma$.
    \item We emphasize that the amplitude of the oscillatory term $\osc(\beta, \lambda)$ grows at most linearly in $\lambda$ independently of the values of $\gamma$ and $\beta$:
    \begin{itemize}
    \item In the rational case~\ref{thm:Asymptotics_part_I_rational} the only term of $\osc(\beta, \lambda)$ that is not bounded is the first one,
    \begin{equation}
        \osc\bigl(\tfrac{\mu}{\nu}, \lambda\bigr) = \frac{\zeta(-\gamma, \{x\})}{(\mu\nu)^{\frac{1+\gamma}{2}}}\lambda +  O(1), \quad \mbox{as }\lambda \to \infty.
      \end{equation}
    \item In the irrational case~\ref{thm:Asymptotics_part_II_irrational} we believe that $\osc(\beta, \lambda)=o(\lambda)$. For $\gamma=0$ it follows that this is the case from the results in~\cite{MarshallSteinerberger}, but we are currently unable to prove this when $\gamma>0$. Whether or not this statement is true will be of little importance in what follows, but if one aims to prove (or disprove) Conjecture~\ref{conj:conjectured_limits} it would most likely be necessary to understand $\osc(\beta, \lambda)$ in greater detail. 
    \end{itemize}
  \end{enumerate}
\end{remark}

For an explicit formula for the coefficients $\alpha_k$ see~\eqref{eq:alpha_coefficients}. For our purposes it will only be important that
\begin{align}\label{eq:low_order_coefficients}
  \alpha_{0} &= \frac{1}{(1+\gamma)(2+\gamma)}, 
  \qquad 
  \alpha_1 = -\frac{(1+2\sigma)\sqrt{\beta}+(1+2\tau)/{\sqrt{\beta}}}{2(1+\gamma)},\\
  \alpha_{2} &= 
  \frac{(1+2\sigma)(1+2\tau)}{4}+ \frac{(1+6\sigma(1+\sigma))\beta+(1+6\tau(1+\tau))/\beta}{12}.
\end{align}
The $\alpha_2$ term will only be important in the case $\sigma=\tau=-1/2$, in which case $\alpha_1=0$ and $\alpha_2= -\frac{1+\beta^2}{24\beta}$.

Heuristically, Theorem~\ref{thm:asymptotics} suggests that Theorems~\ref{thm:Mainconvtheorem},~\ref{thm:Mainconvtheorem_oscillators} and Conjecture~\ref{conj:conjectured_limits} should be true. Essentially, since the first order term is independent of $\beta$ it is reasonable to conjecture that to asymptotically maximize $R_{\sigma, \tau}^\gamma$ one would want to choose $\beta$ to maximize the next order term. The cases where we can prove that an asymptotic maximizer exists is precisely those where: 
\begin{enumerate}[label=(\roman*)]
\item the subleading polynomial term is asymptotically much larger than $\osc(\beta, \lambda)$, and
\item the coefficient of this term is maximized at some $\beta\in \R_\limplus$.
\end{enumerate}
In the harmonic oscillator case, when $\alpha_1=0$, this means that the third term $\alpha_2\lambda^\gamma$ needs to be superlinear, and hence $\gamma>1$.

For the combinations of $\sigma, \tau$ and $\gamma$ in Conjecture~\ref{conj:conjectured_limits} the oscillatory parts of the expansion are of greater importance. It is suitable to consider the renormalized quantity
\begin{equation}\label{eq:periodic limit}
  \frac{R_{\sigma, \tau}^\gamma(\beta, \lambda) - \alpha_0\lambda^{2+\gamma}}{\lambda}.
\end{equation}
If $\sigma, \tau$ and $\gamma$ are as in Conjecture~\ref{conj:conjectured_limits} then in the limit $\lambda\to \infty$~\eqref{eq:periodic limit} converges to a function which is periodic in $\lambda$ and whose period and amplitude depend on $\beta$ (for $\gamma>0$ this follows from Theorem~\ref{thm:asymptotics} and for $\gamma=0$ from~\cite[Lemmas~4 and~5]{MarshallSteinerberger} by a change of variables). It is not unreasonable to believe that one can align these periods to construct a large set of limit points for~$\beta_{\sigma, \tau}^\gamma(\lambda).$ In fact, this is the underlying idea in Marshall and Steinerberger's proof of the conjecture in the the case $\sigma=\tau=\gamma=0$~\cite{MarshallSteinerberger},

From Theorem~\ref{thm:asymptotics} it is not difficult to conclude that any sequence of maximizers of~$R^\gamma_{\sigma, \tau}$ must degenerate when $(\sigma, \tau)\in (-1, \infty)^2 \setminus((-1/2, \infty)^2 \cup\{(-1/2, -1/2)\})$. Indeed, for such shifts the second term of the asymptotic expansion is maximized when $\beta$ tends either to $0$ or $\infty$. Since the expansion is uniform on compact sets this implies that any maximizing sequence must eventually leave all compacts. 

In the case of $H_{\sigma, \tau}$ similar reasoning can be used to conclude that any sequence of maximizers $\beta^H_{\sigma, \tau}$ must degenerate as $t\to 0^\limplus$. Indeed, from Theorem~\ref{thm:asymptotics} and~\eqref{eq:heat_trace_as_Laplace_transform} one finds that
\begin{equation}
  H_{\sigma, \tau}(\beta, t) = \frac{1}{t^2}- \frac{(1+2\sigma)\sqrt{\beta}+ (1+2\tau)/{\sqrt{\beta}}}{2 t} + O(1), \quad \mbox{as }t\to 0^\limplus, 
\end{equation}
where the remainder term is uniform for $\beta$ on compact subsets of $\R_\limplus$, which allows us to argue as above.

\subsection{Higher dimensions}

Using an idea of Laptev~\cite{Laptev2} and the bounds proved in Section~\ref{sec:Prelim} one can reduce the corresponding $d$-dimensional version of the problems considered here to lower dimensional ones. In~\cite{GittinsLarson} this strategy was applied to generalize the results of~\cite{MR3001382, MR3556370, vdBergGittins} to any dimension.

Providing asymptotic expansions similar to those in Theorem~\ref{thm:asymptotics} in higher dimensions is possible using the techniques from~\cite{MR1062904, MR1079775, MR1061661}, see also Section~\ref{sec:Proof_asymptotics}. Naturally the computations in general dimension are more difficult. However, for the cases where one would expect the existence of an asymptotic maximizer the formulas in Theorem~\ref{thm:asymptotics} are more detailed than necessary. For the applications considered, it is sufficient to know the first and second non-vanishing polynomial term, and that the oscillatory part of the expansion is of lower order than the second polynomial term. In the $d$-dimensional case the oscillatory terms will generally be of magnitude $\sim\lambda^{d-1}$. Precise, and uniform, asymptotic expansions to sufficiently low order can be obtained following the argument in Section~\ref{sec:Bounding_Osc}.

\subsection{Structure of the paper}
The remainder of the paper is structured as follows. In Section~\ref{sec:Prelim} we prove a number of bounds for $R_{\sigma, \tau}^\gamma$ which will enable us to exclude that there are sequences of maximizers which degenerate as $\lambda \to \infty$. In Section~\ref{sec:Proof_heat_kernel} we study the problem of maximizing $H_{\sigma, \tau}$ and prove Theorem~\ref{thm:HeatKernel}. Section~\ref{sec:Proof_main_thms} is dedicated to the proofs of Theorems~\ref{thm:Mainconvtheorem} and~\ref{thm:Mainconvtheorem_oscillators}, which will rely on the bounds proved in Section~\ref{sec:Prelim} and Theorem~\ref{thm:asymptotics}. Finally in Section~\ref{sec:Proof_asymptotics} we study the asymptotic behavior of $R_{\sigma, \tau}^\gamma(\beta, \lambda)$, as $\lambda \to \infty$, and prove Theorem~\ref{thm:asymptotics}.


\section{Preliminaries}
\label{sec:Prelim}

Before we continue we need to verify that we can actually talk about maximizers of $R^\gamma_{\sigma, \tau}(\, \cdot\, , \lambda)$ and $H_{\sigma, \tau}(\, \cdot\, , t)$. For $H_{\sigma, \tau}$ it is clear from~\eqref{eq:Htrace_explicit_form} that the maximization problem is well posed, and hence we only need to prove that this is the case for $R_{\sigma, \tau}^\gamma$.

\begin{lemma}\label{lem:existenceAprioriBounds}
  For each $\lambda\geq 0, \gamma>0$ and $\sigma, \tau>-1$ there exists a maximizing value $\beta^\gamma_{\sigma, \tau}(\lambda)$. If $\lambda\le 2\sqrt{(1+\sigma)(1+\tau)}$ then $R^\gamma_{\sigma, \tau}(\beta, \lambda)=0$ for all $\beta>0$, and thus any $\beta$ is a maximizer. If $\lambda > 2\sqrt{(1+\sigma)(1+\tau)}$ then all maximizers satisfy
  \begin{align}
    \beta^\gamma_{\sigma, \tau}(\lambda) 
    &\in \biggl(\frac{(1+\tau)^2}{\lambda^2}, \frac{\lambda^2}{(1+\sigma)^2}\biggr).
  \end{align}
\end{lemma}
Lemma~\ref{lem:existenceAprioriBounds} follows directly from~\cite[Lemma~9]{LaugesenLiu2}, but since our notation is different and the proof is simple we choose to include it.
\begin{proof}[Proof of Lemma~\ref{lem:existenceAprioriBounds}]
  Note first that if we can prove the second part of the lemma, that there are no maximizers outside $\bigl(\tfrac{(1+\tau)^2}{\lambda^2}, \tfrac{\lambda^2}{(1+\sigma)^2}\bigr)$, then the existence of a maximizer is clear by the continuity of $R_{\sigma, \tau}^\gamma(\beta, \lambda)$ as a function of $\beta$. 

  That $R_{\sigma, \tau}^\gamma(\beta, \lambda)=0$ for all $\beta$ if $\lambda\leq 2\sqrt{(1+\sigma)(1+\tau)}$ follows since the inequality
  \begin{equation}\label{eq:Rmean_is_0}
    \lambda-(1+\sigma)\sqrt{\beta}-(1+\tau)/{\sqrt{\beta}}\leq 0, 
  \end{equation}
  holds for all $\lambda\leq 2\sqrt{(1+\sigma)(1+\tau)}$. Similarly,~\eqref{eq:Rmean_is_0} holds if $\beta\leq \frac{(1+\tau)^2}{\lambda^2}$ or $\beta \geq \frac{\lambda^2}{(1+\sigma)^2}$, and thus $R_{\sigma, \tau}^\gamma(\beta, \lambda)=0$ for such $\beta$. However, if $\lambda>2\sqrt{(1+\sigma)(1+\tau)}$ then $R_{\sigma, \tau}^\gamma(\beta^\gamma_{\sigma, \tau}(\lambda), \lambda)\geq R_{\sigma, \tau}^\gamma(\frac{1+\tau}{1+\sigma}, \lambda)>0$, which implies that $\beta\notin\bigl(\tfrac{(1+\tau)^2}{\lambda^2}, \tfrac{\lambda^2}{(1+\sigma)^2}\bigr)$ cannot be a maximizer.
\end{proof}

To conclude that any sequence of maximizers of $R_{\sigma, \tau}^\gamma$, with $\lambda \to \infty$, remains in a compact subset of $\R_\limplus$ we require better control than that provided by Lemma~\ref{lem:existenceAprioriBounds}. When proving that this is in fact the case the following bounds will be useful:


\begin{lemma}\label{lem:1D_bounds}
  We have that:
  \begin{enumerate}[label=(\roman*)]
    \item\label{itm:berezin_1D} For $\sigma\geq -1/2$, 
    \begin{equation}
       \sum_{k\geq 1}(\lambda-(k+\sigma)\sqrt{\beta})_\limplus 
        \leq 
        \frac{\lambda^2}{2\sqrt{\beta}}, \label{eq:1D_Berezin}
    \end{equation}
    for all $\beta>0$ and $\lambda \geq 0$.

    \item\label{itm:general_sigma_1Dlem} For $\sigma>-1/2$ there exist positive constants $c_1, c_2, b_0$ such that
    \begin{align}
      \sum_{k\geq 1}(\lambda-(k+\sigma)\sqrt{\beta})_\limplus 
        &\leq 
        \frac{\lambda^2}{2\sqrt{\beta}}- c_1b\lambda+ c_2b^2  \sqrt{\beta}, \label{eq:1D_3term_ineq}
    \end{align}
    for all $\beta>0, \lambda \geq0$ and $b\in [0, b_0].$

    \item\label{itm:osc_1Dlem} There exist positive constants $c_1, c_2, b_0$ such that
    \begin{equation}\label{eq:1D_osc_bound}
      \sum_{k\geq 1}(\lambda-(k-\tfrac{1}{2})\sqrt{\beta})^2_\limplus \leq \frac{\lambda^3}{3\sqrt{\beta}}-c_1 b\sqrt{\beta}\lambda + c_2 b^{3/2}\beta, 
    \end{equation}
    for all $\beta>0, \lambda \geq 0$ and $b\in [0, b_0]$.
  \end{enumerate}
\end{lemma}

\begin{proof}[{Proof of Lemma~\ref{lem:1D_bounds}}]
We begin by proving parts~\ref{itm:berezin_1D} and~\ref{itm:general_sigma_1Dlem} of the lemma. Clearly~\ref{itm:general_sigma_1Dlem} implies~\ref{itm:berezin_1D} when $\sigma>-1/2$. For $\sigma\geq -1/2$, 
\begin{align}
  \sum_{k\geq 1}(\lambda-(k+\sigma)\sqrt{\beta})_\limplus 
  &= 
    \sum_{k=1}^{\lf \lambda/{\sqrt{\beta}}-\sigma\rf}(\lambda-(k+\sigma)\sqrt{\beta})\\
  &=
    \frac{\lambda^2}{2\sqrt{\beta}}-\frac{1+2\sigma}{2}\lambda+\frac{r-r^2+\sigma+\sigma^2}{2}\sqrt{\beta}, 
    \label{eq:proof_1Dlem_e1}
\end{align}
where $r=\bigl\{\tfrac{\lambda}{\sqrt{\beta}}-\sigma\bigr\}$. Maximizing the right-hand side of~\eqref{eq:proof_1Dlem_e1} with respect to $r\in[0, 1)$ we find
\begin{equation}\label{eq:proof_1Dlem_e2}
\sum_{k\geq 1}(\lambda-(k+\sigma)\sqrt{\beta})_\limplus \leq\frac{\lambda^2}{2\sqrt{\beta}}-\frac{1+2\sigma}{2}\lambda +\frac{(1+2\sigma)^2}{8} \sqrt{\beta}, 
\end{equation}
which implies~\ref{itm:berezin_1D} when $\sigma=-1/2$. Moreover, since the left-hand side of~\eqref{eq:proof_1Dlem_e2} is decreasing in $\sigma$ we find~\ref{itm:general_sigma_1Dlem} with $c_1=1/2, c_2=1/8$ and $b_0=1+2\sigma$.

The proof of part~\ref{itm:osc_1Dlem} is similar:
\begin{align}
  \sum_{k\geq 1}(\lambda-(k-\tfrac12)\sqrt{\beta})_\limplus^2
  &= 
    \sum_{k=1}^{\lf \lambda/{\sqrt{\beta}}+1/2\rf}(\lambda-(k-\tfrac12)\sqrt{\beta})^2\\
  &=
    \frac{\lambda^3}{3\sqrt{\beta}}-\frac{\sqrt{\beta}}{12}\lambda-\frac{r(1-r)(1-2r)}{6}\beta\\
  &\leq
    \frac{\lambda^3}{3\sqrt{\beta}}-\frac{\sqrt{\beta}}{12}\lambda+\frac{\beta}{36\sqrt{3}}, 
\end{align}
where we again maximized in $r=\bigl\{\tfrac{\lambda}{\sqrt{\beta}}+\tfrac12\bigr\}$.

We aim for a bound on the form
\begin{equation}
  \sum_{k\geq 1}(\lambda-(k-\tfrac12)\sqrt{\beta})_\limplus^2\leq \frac{\lambda^3}{3\sqrt{\beta}}-b\sqrt{\beta}\lambda + \frac{2}{3}b^{3/2}\beta.
\end{equation}
The right-hand side is non-negative for $b, \beta>0$ and $\lambda\geq 0$, and hence the bound is trivially true when the left-hand side is zero, i.e.\ for $\lambda \leq \sqrt{\beta}/2$. It thus suffices to prove that
\begin{equation}
  \frac{\lambda^3}{3\sqrt{\beta}}-\frac{\sqrt{\beta}}{12}\lambda+\frac{\beta}{36\sqrt{3}}\leq \frac{\lambda^3}{3\sqrt{\beta}}-b\sqrt{\beta}\lambda + \frac{2}{3}b^{3/2}\beta, 
\end{equation}
when $b$ is small enough and $\lambda \geq \sqrt{\beta}/2$. The above inequality holds for all $\lambda \geq \sqrt{\beta}/2$ if and only if
\begin{equation}
  b\leq 1/12 \quad \mbox{and}\quad -\frac{9}{2}+\sqrt{3}+54b-72b^{3/2}\leq 0, 
\end{equation}
which it is easy to check holds for all $b\in [0, 1/12]$. This completes the proof of~\ref{itm:osc_1Dlem} with $c_1=1, c_2=2/3$ and $b_0=1/12$, and hence the proof of Lemma~\ref{lem:1D_bounds}.
\end{proof}


Based on Lemma~\ref{lem:1D_bounds} we can adapt an idea from~\cite{Laptev2} (see also~\cite{GittinsLarson}) to reduce the proof of a good enough bound for the counting function to a bound for what is essentially a one-dimensional Riesz mean of order $1$.

\begin{lemma}\label{lem:bound_N}
  Fix $\sigma, \tau >-1/2$. There exist positive constants $c_1, c_2, c_3, b_0$ such that
  \begin{equation}
    N_{\sigma, \tau}(\beta, \lambda) \leq \frac{\lambda^2}{2}- c_1 b\frac{1+\beta}{\sqrt{\beta}}\lambda + c_2 b^2\frac{1+\beta^2}{\beta} + c_3(\lambda+1), 
  \end{equation}
  for all $\lambda\geq 0$, $\beta>0$ and $b\in [0, b_0]$.
\end{lemma}

\begin{remark}
  A similar bound appears in~\cite[Propositon~10]{LaugesenLiu2}. However, for $\sigma, \tau$ small the linear term of that bound becomes positive. In what follows it will be essential for this term to be negative, which corresponds to the positivity of $c_1$ in Lemma~\ref{lem:bound_N}.
\end{remark}

\begin{proof}[Proof of Lemma~\ref{lem:bound_N}]
The bound is an easy consequence of Lemma~\ref{lem:1D_bounds}. First observe that for all $\lambda\geq 0$, $\beta >0$ and $\sigma'\in (-1/2, \min\{\sigma, \tau\}]$ we have
\begin{align}
  N_{\sigma, \tau}(\beta, \lambda) 
  &\leq N_{\sigma'}(\beta, \lambda).
\end{align}

A straightforward estimate yields that
\begin{align}
  N_{\sigma'}(\beta, \lambda) 
  &=
  \sum_{k\in\N^2}(\lambda - (k_1+\sigma')\sqrt{\beta}-(k_2+\sigma')/{\sqrt{\beta}})_\limplus^0\\
  &=
  \sum_{k_1\geq 1}\lf (\lambda\sqrt{\beta}-(k_1+\sigma)\beta-\sigma')_\limplus\rf\\
  &\leq
  \sum_{k_1\geq 1} (\lambda\sqrt{\beta}-(k_1+\sigma')\beta+\sigma'_\limminus)_\limplus\\
  &= 
  \sqrt{\beta}\sum_{k_1\geq 1}(\lambda+\sigma'_\limminus/{\sqrt{\beta}}-(k_1+\sigma')\sqrt{\beta})_\limplus.\label{eq:shifter_Rmean}
\end{align}
Applying~\eqref{eq:1D_3term_ineq} of Lemma~\ref{lem:1D_bounds} one obtains that
\begin{align}\label{eq:Nbound1}
  N_{\sigma'}(\beta, \lambda) 
  &\leq 
  \frac{(\lambda+\sigma'_\limminus/{\sqrt{\beta}})^2}{2}-c_1 b\sqrt{\beta}(\lambda+\sigma'_\limminus/{\sqrt{\beta}})+c_2 b^2\beta\\
  &=
  \frac{\lambda^2}{2}- c_1 b\sqrt{\beta}\lambda + c_2 b^2\beta + \frac{\sigma'_\limminus}{\sqrt{\beta}}\lambda + \frac{(\sigma'_\limminus)^2}{2\beta}-c_1 b \sigma'_\limminus.
\end{align}
Arguing as above but switching the roles of $k_1$ and $k_2$ one correspondingly finds that
\begin{equation}\label{eq:Nbound2}
  N_{\sigma'}(\beta, \lambda) 
  \leq
  \frac{\lambda^2}{2}- \frac{c_1 b}{\sqrt{\beta}}\lambda + \frac{c_2 b^2}{\beta} + \sigma'_\limminus\sqrt{\beta}\lambda + \frac{(\sigma'_\limminus)^2}{2}\beta-c_1 b \sigma'_\limminus.
\end{equation}
Together these two bounds imply that
\begin{equation}\label{eq:Nbound3}
  N_{\sigma'}(\beta, \lambda) \leq 
  \frac{\lambda^2}{2}- \frac{c_1 b}{2}\frac{1+\beta}{\sqrt{\beta}}\lambda + c_2 b^2\frac{1+\beta^2}{\beta} + \sigma'_\limminus\lambda + \frac{(\sigma'_\limminus)^2}{2}-c_1 b \sigma'_\limminus.
\end{equation}
Indeed, for $\beta\geq 1$ the right-hand side of~\eqref{eq:Nbound3} is larger than that of~\eqref{eq:Nbound1}, and for $\beta\leq 1$ larger than that of~\eqref{eq:Nbound2}. This completes the proof of the claimed bound with constants related to those in Lemma~\ref{lem:1D_bounds}.
\end{proof}


In the case of the harmonic oscillators we prove the following lemma, which will play the same role as Lemma~\ref{lem:bound_N} in what follows.

\begin{lemma}\label{lem:2D_osc_bound}
  There exist positive constants $c_1, c_2, b_0$ such that
  \begin{equation}
    R^1_{\limminus1/2}(\beta, \lambda) \leq \frac{\lambda^3}{6}- c_1 b \frac{1+\beta^2}{\beta}\lambda + c_2 b^{3/2}\frac{1+\beta^3}{\beta^{3/2}}, 
  \end{equation}
  for all $\beta>0, \lambda \geq 0$ and $b\in [0, b_0]$.
\end{lemma}

\begin{proof}[Proof of Lemma~\ref{lem:2D_osc_bound}]
  Again the lemma is a simple consequence of Lemma~\ref{lem:1D_bounds}.
  Applying first~\eqref{eq:1D_Berezin} and then~\eqref{eq:1D_osc_bound} we find that
  \begin{align}
    R_{\limminus1/2}^1(\beta, \lambda) 
    &= 
    \sum_{k\in\N^2}(\lambda-(k_1-\tfrac12)\sqrt{\beta}-(k_2-\tfrac12)/{\sqrt{\beta}})_\limplus\\
    &\leq
    \frac{\sqrt{\beta}}{2}\sum_{k_1\geq 1}(\lambda-(k_1-\tfrac12)\sqrt{\beta})^2_\limplus\\
    &\leq
    \frac{\lambda^3}{6}-\frac{c_1b}{2}\beta\lambda+\frac{c_2 b^{3/2}}{2}\beta^{3/2}.
  \end{align}
  Arguing identically but switching the roles of $k_1$ and $k_2$ we find that 
  \begin{equation}
    R_{\limminus1/2}^1(\beta, \lambda) \leq \frac{\lambda^3}{6}-\frac{c_1b}{2}\beta^{-1}\lambda+\frac{c_2 b^{3/2}}{2}\beta^{-3/2}.
  \end{equation}
  Taking the average of the two bounds completes the proof of Lemma~\ref{lem:2D_osc_bound}.
\end{proof}


Combining Lemmas~\ref{lem:bound_N} and~\ref{lem:2D_osc_bound} with the Aizenman--Lieb Identity~\eqref{eq:Aizenman_Lieb} one finds the following.
\begin{corollary}\label{cor:three_term_bounds_gamma}
We have that:
\begin{enumerate}[label=(\roman*)]
   \item For $\sigma, \tau >-1/2$ and $\gamma>0$, there exist positive constants $c_1, c_2, c_3, b_0$ such that
  \begin{equation}\label{eq:general_gamma}
    R_{\sigma, \tau}^\gamma(\beta, \lambda) \leq \frac{\lambda^{2+\gamma}}{(1+\gamma)(2+\gamma)}-c_1 b\frac{1+\beta}{\sqrt{\beta}}\lambda^{1+\gamma}+c_2 b^2\frac{1+\beta^2}{\beta}\lambda^\gamma+c_3(\lambda+1)\lambda^\gamma, 
  \end{equation}
  for all $\beta>0, \lambda \geq 0$ and $b\in [0, b_0]$.

  \item For $\gamma \geq 1$ there exist positive constants $c_1, c_2, b_0$ such that
  \begin{equation}\label{eq:osc_gamma}
    R^\gamma_{\limminus1/2}(\beta, \lambda) \leq \frac{\lambda^{2+\gamma}}{(1+\gamma)(2+\gamma)}- c_1 b \frac{1+\beta^2}{\beta}\lambda^{\gamma} + c_2 b^{3/2}\frac{1+\beta^3}{\beta^{3/2}}\lambda^{\gamma-1}, 
  \end{equation}
  for all $\beta>0, \lambda \geq 0$ and $b\in [0, b_0]$.
 \end{enumerate}
\end{corollary}

\begin{remark}
  We note that the proofs above lift without much work to the corresponding $d$-dimensional problem. Using again the Aizenman--Lieb Identity~\eqref{eq:Aizenman_Lieb} one finds a version of~\eqref{eq:1D_Berezin} for higher order Riesz means. When $\gamma\geq 1$ one can follow the lifting argument of~\cite{Laptev2} (used in a similar context in~\cite{GittinsLarson}): use the corresponding one-term bound to bound the first $d-1$ sums and then a bound similar to~\eqref{eq:1D_3term_ineq} to bound the final sum. For the case of the counting function one can mimic~\eqref{eq:shifter_Rmean} reducing the problem to bound a Riesz mean of order one where the spectral parameter $\lambda$ is slightly shifted. Bounding this Riesz mean can be carried out as described above.
\end{remark}


\section{Proof of Theorem~\ref{thm:HeatKernel}}
\label{sec:Proof_heat_kernel}
We now turn to the problem of maximizing $H_{\sigma, \tau}$. Due to the fact that we have a closed expression for $H_{\sigma, \tau}$ this is reduced to a maximization problem in one real variable. However solving this problem still turns out to be rather tedious.

Since
\begin{align}
  H_{\sigma, \tau}(\beta, t) 
  &=
   \frac{e^{-t (\sigma\sqrt{\beta} +\tau/{\sqrt{\beta}})}}{(e^{t\sqrt{\beta}}-1)(e^{t/{\sqrt{\beta}}}-1)}
\end{align}
is non-negative, continuous in $\beta$ and $\lim_{\beta\to0}H_{\sigma, \tau}(\beta, t)=\lim_{\beta\to\infty}H_{\sigma, \tau}(\beta, t)=0$ for all $t>0$ and $\sigma, \tau>-1$, it follows that there is at least one maximizing $\beta$ for each $t$.

Set $x=\sqrt{\beta}$ and note that
\begin{equation}
  H_{\sigma, \tau}(x^2, t)=
  \frac{e^{-t((\sigma+1/2)x+(\tau+1/2)/x)}}{4t^2}\frac{tx/2}{\sinh(tx/2)}\frac{t/(2x)}{\sinh(t/(2x)}.
\end{equation}
By the monotonicity of the logarithm we can equivalently consider maximizing $\log(H_{\sigma, \tau})$:
\begin{align}
  \log(H_{\sigma, \tau}(x^2, t))
  &=
   -t((\sigma+1/2)x+(\tau+1/2)/x)\\
   &\quad-\log\biggl(\frac{\sinh(tx/2)}{tx/2}\biggr)-\log\biggl(\frac{\sinh(t/(2x))}{t/(2x)}\biggr)-\log(4t^2).
\end{align}
By recalling that $\log(\sinh(x)/x)$ is increasing and strictly convex on $\R_\limplus$ it follows that
\begin{equation}
  \frac{\partial^2 }{\partial x^2}\log(H_{\sigma, \tau}(x^2, t))< -t\frac{\tau+1/2}{2x^3}. 
\end{equation}
Hence, if $\tau\geq-1/2$ the function $\log(H_{\sigma, \tau}(x^2, t))$ is concave in $x$. Since $\log(H_{\sigma, \tau}(x^2, t))$ also tends to $-\infty$ when $x\to 0$ or $\infty$ it has a unique maximum. Since $H_{\sigma, \tau}(x^2, t)=H_{\tau, \sigma}(1/x^2, t)$ we can conclude that the same is true if instead $\sigma \geq -1/2$. Moreover, when $\sigma=\tau\geq -1/2$ the symmetry $H_{\sigma}(x^2, t)=H_\sigma(1/x^2, t)$ implies that $x=1$ must be the unique maximizer.

As the function $x \mapsto \log(H_{\sigma, \tau}(x^2, t))$ is smooth any maximizing $x^*(t)$ must satisfy
\begin{equation}\label{eq:log_heat_trace_derivative}
  \frac{\partial }{\partial x}\log(H_{\sigma, \tau}(x^2, t))
  =
  -\frac{t}{2x^2}\Bigl[\Bigl(1+2\sigma+\coth\Bigl(\frac{tx}{2}\Bigr)\Bigr)x^2-\Bigl(1+2\tau+\coth\Bigl(\frac{t}{2x}\Bigr)\Bigr)\Bigr]=0.
\end{equation}
When $t\to 0^\limplus$ it is easy to see that this equation
has a solution which converges to $\sqrt{\frac{1+2\tau}{1+2\sigma}}$. Similarly when $t\to \infty$ we see that there is a solution converging to $\sqrt{\frac{1+\tau}{1+\sigma}}.$ When $\max\{\sigma, \tau\}> -1/2$ this concludes the proof of the theorem since we know that the solution is unique.

When $\sigma$ and $\tau$ are both less than $-1/2$ maximizers are no longer necessarily unique when $t$ is small. However, when $t \to \infty$ any sequence of maximizers converges. If there is some solution $x^*(t)$ of~\eqref{eq:log_heat_trace_derivative} which remains in a compact subset of $\R_\limplus$ as $t\to \infty$, we must have that
\begin{equation}
  \lim_{t\to \infty} x^*(t) = \sqrt{\frac{1+\tau}{1+\sigma}},
 \end{equation} 
 since otherwise the expression in the brackets is bounded away from zero when $t$ is large enough.

 What remains is to conclude that there can be no maximizers which degenerate, thus implying that the asymptotically stable stationary point is indeed an asymptotic maximizer. Since $H_{\sigma, \tau}(x^2, t)=H_{\tau, \sigma}(1/x^2, t)$ any sequence of maximizers tending to infinity as $t\to \infty$ implies the existence of a sequence of maximizers tending to zero for the problem where $\sigma$ and~$\tau$ have been interchanged. Therefore it is sufficient to show that we cannot have maximizers degenerating to zero.

 Assume for contradiction that we have a sequence of maximizers $x^*=x^*(t)$ such that $\lim_{t\to \infty}x^*=0$. Since the factor in front of the parenthesis is non-zero,~\eqref{eq:log_heat_trace_derivative} implies that
 \begin{equation}
   \lim_{t\to \infty}\coth\Bigl(\frac{t x^*(t)}{2}\Bigr)x^*(t)^2=2+2\tau.
 \end{equation}
But this is a contradiction since
\begin{equation}
0\leq \coth\Bigl(\frac{t x^*(t)}{2}\Bigr)x^*(t)^2\leq \Bigl(1+\frac{2}{t x^*(t)}\Bigr)x^*(t)^2 \to 0 \quad \mbox{as } t\to \infty,
\end{equation}
which completes the proof of Theorem~\ref{thm:HeatKernel}.


\section{Proof of Theorems~\ref{thm:Mainconvtheorem} and~\ref{thm:Mainconvtheorem_oscillators}}
\label{sec:Proof_main_thms}

We now turn our attention to the main results of the paper, namely Theorems~\ref{thm:Mainconvtheorem} and~\ref{thm:Mainconvtheorem_oscillators}. As the proofs of the two theorems are essentially identical we will write out only the former in detail. The main idea is to combine the bounds in Corollary~\ref{cor:three_term_bounds_gamma} with Theorem~\ref{thm:asymptotics} following the strategy of~\cite{MR3001382}, with some modifications resembling those in~\cite{GittinsLarson}.

Fix $\sigma, \tau>-1/2$ and $\gamma>0$. For notational convenience we will write $R(\beta, \lambda)=R_{\sigma, \tau}(\beta, \lambda)$, $\beta=\beta^\gamma_{\sigma, \tau}(\lambda)$ and $\beta^*=\frac{1+2\tau}{1+2\sigma}$ throughout the proof.

By the maximality of $\beta$ and~\eqref{eq:general_gamma} of Corollary~\ref{cor:three_term_bounds_gamma} we have that
\begin{equation}
  R(\beta^*\!, \lambda)\leq R(\beta, \lambda) \leq \frac{\lambda^{2+\gamma}}{(1+\gamma)(2+\gamma)}-c_1 b\frac{1+\beta}{\sqrt{\beta}}\lambda^{1+\gamma}+c_2 b^2\frac{1+\beta^2}{\beta}\lambda^\gamma+c_3(\lambda+1)\lambda^\gamma.
\end{equation}
Using the asymptotic expansion of the left-hand side given by Theorem~\ref{thm:asymptotics}, rearranging and using that $\frac{1+\beta^2}{1+\beta}\geq1+\beta$ we find
\begin{equation}\label{eq:beta_bound_mainproof}
  c_1 b \frac{1+\beta}{\sqrt{\beta}}\Bigl(1- b \frac{c_2 (1+\beta)}{c_1 \sqrt{\beta}\lambda} \Bigr) \leq C + O(\lambda^{-\min\{1, \gamma\}}), 
\end{equation}
as $\lambda\to \infty.$

From Lemma~\ref{lem:existenceAprioriBounds} we know that $\frac{1+\beta}{\sqrt{\beta}\lambda}\leq \frac{1}{1+\tau}+\frac{1}{1+\sigma}$, and hence we can choose $b$ small enough so that the left-hand side of~\eqref{eq:beta_bound_mainproof} is positive. Therefore we conclude that
\begin{equation}
  \limsup_{\lambda \to \infty} \frac{1+\beta}{\sqrt{\beta}} \leq C, 
\end{equation}
and hence $\beta=\beta_{\sigma, \tau}^\gamma(\lambda)$ remains uniformly bounded away from zero and infinity.

As we now know that all maximizers are contained in a compact subset of $\R_\limplus$ we can use Theorem~\ref{thm:asymptotics} to expand both sides of the inequality $R(\beta^*\!, \lambda)\leq R(\beta, \lambda)$ with remainder terms independent of $\beta$. After rearranging this yields that
\begin{equation}\label{eq:asym_bound_beta}
    (1+2\sigma)\sqrt{\beta}+(1+2\tau)/{\sqrt{\beta}}\leq
    (1+2\sigma)\sqrt{\beta^*}+(1+2\tau)/{\sqrt{\beta^*}}
     + O(\lambda^{-\min\{\gamma, 1\}}).
\end{equation}  
Since $\beta^*$ is the unique minimizer of the function $x \mapsto (1+2\sigma)\sqrt{x}+(1+2\tau)/{\sqrt{x}}$ and the remainder term is independent of $\beta$,
\eqref{eq:asym_bound_beta} implies that
\begin{equation}
  \beta= \beta^* + o(1)\quad \mbox{as } \lambda \to \infty, 
\end{equation}
which concludes the proof of Theorem~\ref{thm:Mainconvtheorem}.

The proof of Theorem~\ref{thm:Mainconvtheorem_oscillators} is almost identical with the only change being the application of~\eqref{eq:osc_gamma} instead of~\eqref{eq:general_gamma} in the first part of the proof.


\section{Proof of Theorem~\ref{thm:asymptotics}}
\label{sec:Proof_asymptotics}

What remains is to prove Theorem~\ref{thm:asymptotics}. The calculations follow those of Helffer and Sj\"ostrand in~\cite{MR1062904} for the isotropic harmonic oscillator $\beta=1$ and $\sigma=\tau=-1/2$ (see also~\cite{MR1061661, MR1079775}). The key idea is to use the Laplace transform to rewrite $R_{\sigma, \tau}^\gamma$ as an integral which opens up for use of the residue theorem. For any $c>0$, 
\begin{align}
  R^\gamma_{\sigma, \tau}(\beta, \lambda) &= \sum_{k\in \N^2}(\lambda-(k_1+\sigma)\sqrt{\beta}-(k_2+\tau)/{\sqrt{\beta}})_\limplus^\gamma\\
  &= \sum_{k\in \N^2}\frac{\Gamma(1+\gamma)}{2\pi i}\int_{c-i\infty}^{c+i\infty}e^{t(\lambda-(k_1+\sigma)\sqrt{\beta}-(k_2+\tau)/{\sqrt{\beta}})}t^{-1-\gamma}\, dt\\
  &=\frac{\Gamma(1+\gamma)}{2\pi i} \int_{c-i\infty}^{c+i\infty}\frac{e^{t(\lambda- \sigma\sqrt{\beta}-\tau/{\sqrt{\beta}})}}{(e^{t\sqrt{\beta}}-1)(e^{t/{\sqrt{\beta}}}-1)}t^{-1-\gamma}\, dt.
\end{align}
The integrand in the last expression is a meromorphic function of $t$ outside of $(-\infty, 0]$, with poles at $t=2\pi i k\sqrt{\beta}$ and $t=2\pi i l/{\sqrt{\beta}}$, for $k, l\in \Z\setminus\{0\}$. If $\beta$ is irrational all of these poles are simple. If $\beta\in \Q$ say $\beta=\tfrac\mu\nu$, with $\gcd(\mu, \nu)=1$, then there are degree-two poles whenever $k, l$ are related by $\mu k = \nu l$. The remaining poles remain simple. That is, degree-two poles at $t=2\pi i \sqrt{\mu \nu}\, m$ for $m\in \Z\setminus\{0\}$, 
and simple poles at $t=2\pi i \sqrt{\frac{\mu}{\nu}} k_1$ and $t=2\pi i \sqrt{\frac{\nu}{\mu}} k_2$ for $k_1, k_2\in \Z\setminus\{0\}$ such that $\beta k_1= \frac{\mu k_1}{\nu}\notin \Z$ and $\frac{k_2}{\beta}=\frac{\nu k_2}{\mu}\notin \Z$.

Letting $f(t)=\frac{e^{t(\lambda- \sigma\sqrt{\beta}-\tau/{\sqrt{\beta}})}}{(e^{t\sqrt{\beta}}-1)(e^{t/{\sqrt{\beta}}}-1)}t^{-1-\gamma}$ and formally using the residue theorem, one would obtain that
\begin{align}\label{eq:Residue+Int}
  R^\gamma_{\sigma, \tau}(\beta, \lambda) &= \Gamma(1+\gamma)\sum_{t\in\P(f)}\Res(f, t) + \frac{\Gamma(1+\gamma)}{2\pi i}\int_{\Gamma_1}f(t)\, dt, 
\end{align}
where $\P(f)$ denotes the poles of $f$ and $\Gamma_1$ is a contour oriented counter-clockwise which encircles the negative real axis but none of the poles of $f$. However, to make this rigorous we need that the sum of residues is absolutely convergent. We shall prove  that this is the case when $\beta\in \Q_\limplus$ but possibly not when $\beta \notin \Q_\limplus$.

It is no big surprise that the contributions to the asymptotic expansion coming from the residues is the most complicated part to analyse. It is this part which accounts for the oscillatory terms in the expansion and the number theoretic dependence on $\beta$. In contrast the integral over the contour $\Gamma_1$ has an asymptotic expansion in $\lambda$ to arbitrary order as $\lambda$ tends to infinity.

The proof will be split into two parts, first treating $\beta\in \Q_\limplus$ and then $\beta \notin \Q_\limplus$. Much of the work done in the first case will turn out to be useful also in the second.

\subsection{Rational \texorpdfstring{$\beta$}{beta}}

In this case it turns out that the use of the residue theorem above is justified. This will be verified once we prove that the sum of residues is absolutely convergent. However, we begin by studying the non-oscillatory part of the expansion, that is, the contribution from the contour integral in~\eqref{eq:Residue+Int}.

\subsection*{Non-oscillatory part}
Let $\eps \in (0, \min\{\sqrt{\beta}, 1/{\sqrt{\beta}}\}]$, and let $\Gamma_1=\Gamma_\limminus \cup \Gamma_0 \cup \Gamma_\limplus$ with 
\begin{align}
  \Gamma_\limpm &= (-\infty\pm i0, -\eps \pm i0], \\
  \Gamma_0&=\eps e^{i\theta}, \quad \theta \in (-\pi, \pi).
\end{align}
For $\lambda > \sigma\sqrt{\beta}+\tau/{\sqrt{\beta}}$ and any $\eps\in (0, 1)$, we see that 
\begin{equation}
  \Bigl|\int_{\Gamma_\pm}f(t)\, dt\Bigr| \leq \frac{e^{\eps(\sigma \sqrt{\beta}+\tau/{\sqrt{\beta}})}}{\gamma (e^{-\sqrt{\beta}}-1)(e^{-1/{\sqrt{\beta}}}-1)} e^{-\eps\lambda}\eps^{-2-\gamma}.
\end{equation}

Returning to the integral over $\Gamma_0$, 
\begin{equation}
  \int_{\Gamma_0} f(t)\, dt = \int_{\Gamma_0}\frac{e^{t\lambda}}{ t^{3+\gamma}} \frac{t^2 e^{-t(\sigma\sqrt{\beta}+\tau/{\sqrt{\beta}})}}{(e^{t\sqrt{\beta}}-1)(e^{t/{\sqrt{\beta}}}-1)}\, dt.
\end{equation}
For small enough $\eps$ and any $M\in \N$ we have a uniform expansion
\begin{equation}
  \frac{t^2 e^{-t(\sigma\sqrt{\beta}+\tau/{\sqrt{\beta}})}}{(e^{t\sqrt{\beta}}-1)(e^{t/{\sqrt{\beta}}}-1)}= \sum_{k=0}^{M-1} a_k(\beta, \sigma, \tau) t^k + O(t^M), 
\end{equation}
where the implicit constant is uniform for $\beta$ in compact subsets of $\R_\limplus$. The $a_k(\beta, \sigma, \tau)$ are explicitly given by
\begin{equation}
  a_k(\beta, \sigma, \tau)=\sum_{l=0}^k \frac{(-1)^l}{l!}(\sigma\sqrt{\beta}+\tau/{\sqrt{\beta}})^lb_{k-l}(\beta), 
\end{equation}
where the $b_k(\beta)$ are the coefficients in the expansion 
\begin{equation}
  \frac{t^2}{(e^{t\sqrt{\beta}}-1)(e^{t/{\sqrt{\beta}}}-1)}= \sum_{k=0}^{M-1} b_k(\beta) t^k + O(t^M).
\end{equation}
The first few coefficients are given by
\begin{align}
  b_0(\beta) &= 1, &
  b_1(\beta) &=-\frac{1+\beta}{2 \sqrt{\beta }}, &
  b_2(\beta) &=\frac{1+3\beta+\beta^2}{12 \beta }, \\
  b_3(\beta) &=-\frac{1+\beta}{24 \sqrt{\beta }}, &
  b_4(\beta) &=-\frac{1-5\beta^2+\beta ^4}{720 \beta ^2}, &
  b_5(\beta) &=\frac{1+\beta ^3}{1440 \beta ^{3/2}}.
\end{align}

Thus we find that
\begin{align}
  \int_{\Gamma_0}f(t)\, dt 
  &= 
  \sum_{k=0}^{M-1}a_k(\beta, \sigma, \tau)\int_{\Gamma_0}e^{t \lambda}t^{k-3-\gamma}\, dt+ e^{\eps \lambda}O(\eps^{M-2-\gamma})\\
  &=
  \sum_{k=0}^{M-1}a_k(\beta, \sigma, \tau)\int_{\Gamma_1}e^{\lambda t}t^{k-3-\gamma}\, dt+ e^{\eps \lambda}O(\eps^{M-2-\gamma})+ e^{-\eps\lambda}O(\eps^{-2-\gamma}), 
\end{align}
where we used that
\begin{equation}
  \int_{\eps}^\infty e^{-\lambda t}t^{k-3-\gamma}\, dt \leq \sup_{t\geq\eps} (e^{-\lambda t}t^{k})\int_{\eps}^\infty t^{-3-\gamma}\, dt = \frac{e^{-\eps \lambda} \eps^{k-2-\gamma}}{2+\gamma}, 
\end{equation}
provided $\eps \lambda\geq k$.

Recall Hankel's integral representation for the reciprocal $\Gamma$ function~\cite[eq.~5.9.2]{NIST}:
\begin{equation}
  \frac{1}{\Gamma(z)}= \frac{1}{2\pi i}\int e^{t}t^{-z}\, dt, 
\end{equation}
where the integral is over a contour which encircles the origin in the positively oriented direction, beginning and returning to $-\infty$ while respecting the branch cut along the negative real axis. By a change of variables we find that $\int_{\Gamma_1} e^{\lambda t} t^{k-3-\gamma} \, dt= \frac{2\pi i \lambda^{2+\gamma-k}}{\Gamma(3+\gamma-k)}$.

Therefore we conclude that
\begin{align}
  \frac{\Gamma(1+\gamma)}{2\pi i}\int_{\Gamma_1}f(t)\, dt &= \sum_{k=0}^{M-1}a_k(\beta, \sigma, \tau)\frac{\Gamma(1+\gamma)}{\Gamma(3+\gamma-k)}\lambda^{2-k+\gamma}\\
  &\quad +e^{-\eps \lambda}O(\eps^{-2-\gamma})+e^{\eps \lambda}O(\eps^{M-2-\gamma}).
\end{align}
Choose $\eps=\eps(\lambda)$ to solve $e^{-\lambda \eps}=\eps^{M/2}$. For large enough $\lambda$ this choice satisfies the requirements above and the error terms become
\begin{equation}
  O(\eps(\lambda)^{M/2-2-\gamma})= o(\lambda^{-M/2+2+\gamma+\delta}), \quad \forall \delta>0, 
\end{equation}
since $\eps(\lambda)=O(\log(\lambda)/\lambda)=o(\lambda^{-1+\delta})$ for any $\delta>0$.

Moving unnecessary parts into the error term, we have for any $M'\in \N$ and $\delta>0$ that
\begin{equation}
  \frac{\Gamma(1+\gamma)}{2\pi i}\int_{\Gamma_1}f(t)\, dt 
  = \sum_{k=0}^{M'+1} \alpha_{k}\lambda^{2-k+\gamma}
      +o(\lambda^{-M'+\gamma+\delta}), 
\end{equation}
where
\begin{align}\label{eq:alpha_coefficients}
  \alpha_k(\beta, \sigma, \tau, \gamma) = a_k(\beta, \sigma, \tau)\frac{\Gamma(1+\gamma)}{\Gamma(3+\gamma-k)}.
\end{align}

\subsection*{Oscillatory part}
We now turn our attention to the sum of residues
\begin{equation}
  \Gamma(1+\gamma)\sum_{t\in\P(f)}\Res(f, t).
\end{equation}

\subsubsection*{Simple poles} If $t=2\pi i k\sqrt{\beta}$, with $k\in \Z\setminus\{0\}$ such that $\beta k \notin \Z$, then it is straightforward to calculate the residue of $f$ at $t$, yielding:
\begin{equation}
  \Res(f, 2\pi i k\sqrt{\beta})= 
  \beta^{-\gamma/2}\frac{e^{2\pi i k(\lambda\sqrt{\beta}-\sigma\beta -\tau)}}{(2\pi i k)^{1+\gamma}(e^{2\pi i k\beta}-1)}.
\end{equation}
If instead $t=2\pi i k/{\sqrt{\beta}}$, with $k\in \Z\setminus\{0\}$ such that $k/\beta \notin \Z$, then an almost identical calculation leads to:
\begin{equation}
  \Res(f, 2\pi i k/{\sqrt{\beta}}) = 
  \beta^{\gamma/2}\frac{e^{2\pi i k(\lambda/{\sqrt{\beta}}-\sigma-\tau/\beta)}}{(2\pi i k)^{1+\gamma}(e^{2\pi i k/\beta}-1)}.
\end{equation}

Let $x_1=\lambda \sqrt{\beta}-\sigma\beta-\tau$ and $x_2=\lambda/{\sqrt{\beta}}-\sigma-\tau/\beta$. Combining the contributions from $k$ and $-k$ one obtains that
\begin{align}
  \sum_{t\in \P_1} \Res(f, t) 
  &=
  \frac{\beta^{-\gamma/2}}{(2\pi)^{1+\gamma}}\sum_{\substack{k\in \N\\ \beta k \notin \N}} \frac{1}{k^{1+\gamma}}
  \biggl(\frac{e^{2\pi i k x_1}}{e^{i\pi(1+\gamma)/2}(e^{2\pi i k \beta}-1)}
  + 
  \frac{e^{-2\pi i k x_1}}{e^{-i\pi(1+\gamma)/2}(e^{-2\pi i k \beta}-1)}
  \biggr) \\
  &+
   \frac{\beta^{\gamma/2}}{(2\pi)^{1+\gamma}}\sum_{\substack{k\in \N\\ k/\beta \notin \N}} \frac{1}{k^{1+\gamma}}
   \biggl(\frac{e^{2\pi i k x_2}}{e^{i\pi(1+\gamma)/2}(e^{2\pi i k/\beta}-1)}
  + 
  \frac{e^{-2\pi i k x_2}}{e^{-i\pi(1+\gamma)/2}(e^{-2\pi i k/\beta}-1)}
  \biggr)\\[-12pt]
  &=
  \frac{\beta^{-\gamma/2}}{(2\pi)^{1+\gamma}}\sum_{\substack{k\in \N\\ \beta k \notin \N}} \frac{1}{k^{1+\gamma}}
  \biggl(
  \frac{\sin(\pi k (2x_1-\beta)-\frac{\pi}{2}(1+\gamma))}{\sin(\pi k\beta)}
  \biggr)\\
  &+
  \frac{\beta^{\gamma/2}}{(2\pi)^{1+\gamma}}\sum_{\substack{k\in \N\\ k/\beta \notin \N}} \frac{1}{k^{1+\gamma}}
  \biggl(
  \frac{\sin(\pi k (2 x_2-1/\beta)-\frac{\pi}{2}(1+\gamma))}{\sin(\pi k/\beta)}
  \biggr), 
\end{align}
where $\P_1$ denotes the simple poles of $f$.

Let $\beta=\tfrac\mu\nu$, with $\gcd(\mu, \nu)=1$, we shall show that the first of the above series is absolutely convergent, the second can be treated identically. Since $\gcd(\mu, \nu)=1$ we have that $\frac{\mu}{\nu}k\notin\N$ if and only if $\frac{k}{\nu}\notin \N$. We find that
\begin{align}
  \frac{\nu^{\gamma/2}}{\mu^{\gamma/2}}\sum_{\substack{k\in \N\\ k/\nu \notin \N}} 
  \Bigl|\frac{\sin(\pi k(2x_1-\mu/\nu)-\frac{\pi}{2}(1+\gamma))}{k^{1+\gamma}\sin(\pi k \mu/\nu)}\Bigr|
  &\leq 
  \frac{\nu^{\gamma/2}}{\mu^{\gamma/2}}\sum_{\substack{k\in \N\\ k/\nu \notin \N}} \frac{1}{k^{1+\gamma}}\frac{1}{|{\sin(\pi k \mu/\nu)}|}\\
  &=
  \frac{\nu^{\gamma/2}}{\mu^{\gamma/2}}\sum_{l=1}^{\nu-1}\sum_{j=0}^{\infty}\frac{1}{(j\nu + l)^{1+\gamma}}\frac{1}{|{\sin(\pi l\mu/\nu)}|}\\
  &\leq
  \frac{\nu^{\gamma/2}}{\mu^{\gamma/2}}\sum_{l=1}^{\nu-1}\frac{1}{|{\sin(\pi l\mu/\nu)}|}\biggl(\frac{1}{l^{1+\gamma}}+\sum_{j=1}^{\infty}\frac{1}{(j\nu)^{1+\gamma}}\biggr)\\[-14pt]
  &=
  \frac{\nu^{\gamma/2}}{\mu^{\gamma/2}}\sum_{l=1}^{\nu-1}\frac{1}{|{\sin(\pi l\mu/\nu)}|}\biggl(\frac{1}{l^{1+\gamma}}+\frac{\zeta(1+\gamma)}{\nu^{1+\gamma}}\biggr), 
\end{align}
implying that the series is absolutely convergent.

\subsubsection*{Degree-two poles}

When $\beta=\tfrac\mu\nu$ then $f$ has poles of degree two at $t=2\pi i \sqrt{\mu \nu}\, k$, for $k\in \Z\setminus \{0\}$. The residues at these poles can be calculated:
\begin{equation}
   \Res(f, 2\pi i \sqrt{\mu \nu}\, k) 
  = 
    \frac{ e^{2 i \pi  k (\lambda  \sqrt{\mu  \nu }-\mu  \sigma -\nu  \tau )} (\gamma -2 \pi i k (\lambda  \sqrt{\mu  \nu }- \mu  (\sigma +\tfrac12) - \nu (\tau +\tfrac12) )+1)}{(2\pi)^{\gamma +2}(i k)^{\gamma } k^2 (\mu  \nu)^{1+\gamma/2} }.
\end{equation}
It is clear that the sum of these residues is absolutely convergent, which validates our use of the residue theorem in~\eqref{eq:Residue+Int} in the case of rational $\beta$. 

Letting $x_3=\{\sqrt{\mu \nu}\, \lambda-\mu \sigma-\nu\tau\}$ we find that
\begin{align}
  \sum_{t\in \P_2}& \Res(f, t) 
  =
  -\frac{(1+\gamma)}{(\mu \nu)^{1+\gamma/2}}\biggl[e^{-\frac{i\pi}{2}(2+\gamma)}\sum_{k=1}^\infty 
      \frac{e^{2\pi i k x_3}}{(2\pi k)^{2+\gamma}}+e^{\frac{i\pi}{2}(2+\gamma)}\sum_{k=1}^\infty \frac{e^{2\pi i k (1-x_3)}}{(2\pi k)^{2+\gamma}}\biggr]\\
  &\ +
  \frac{\lambda \sqrt{\mu \nu}-\mu\bigl(\sigma+\tfrac12\bigr)-\nu\bigl(\tau+\tfrac12\bigr)}{(\mu \nu)^{1+\gamma/2}}
  \biggl[
  e^{-\frac{i\pi}{2}(1+\gamma)}\sum_{k=1}^\infty \frac{e^{2\pi i kx_3}}{(2\pi k)^{1+\gamma}}+e^{\frac{i\pi}{2}(1+\gamma)}\sum_{k=1}^\infty \frac{e^{2\pi i k(1-x_3)}}{(2\pi k)^{1+\gamma}}
  \biggr]\\
   &\hspace{42.5pt}
   =
  \frac{\lambda \zeta(-\gamma, x_3)-\zeta(-1-\gamma, x_3)/\!{\sqrt{\mu\nu}}
  -\bigl(\bigl(\sigma+\tfrac12\bigr){\scriptstyle\sqrt{\tfrac\mu\nu}}+\bigl(\tau+\tfrac12\bigr)/\!{\scriptstyle\sqrt{{\tfrac\mu\nu}}}\bigr)
  \zeta(-\gamma, x_3)}{(\mu \nu)^{(1+\gamma)/2}\Gamma(1+\gamma)}, \\[-30pt]
\end{align}
where we made use of~\cite[eq.~25.12.13]{NIST}, and $\P_2$ denotes the set of degree-two poles of $f$.

\subsection{Irrational \texorpdfstring{$\beta$}{beta}}

For $\beta\in \R_\limplus{\setminus}\Q$ the calculation leading to the precise asymptotic expansion of $R^\gamma_{\sigma, \tau}$ is slightly more complicated. The complication stems from the fact that we do not know if the sum of residues in~\eqref{eq:Residue+Int} is absolutely convergent. Hence we cannot justify our use of the residue theorem as above. However, by choosing a $\lambda$-dependent contour where we only use the residue theorem for bounded contours one can obtain the desired result.

Fix $\lambda > \sigma \sqrt{\beta}+\tau/{\sqrt{\beta}}$. From the residue theorem we find that for $c>0$ and $\Lambda >1$ to be chosen later
\begin{align}
  R^\gamma_{\sigma, \tau}(\beta, \lambda) 
  &=
  \frac{\Gamma(1+\gamma)}{2\pi i}\int_{c-i\infty}^{c+i\infty} f(t)\, dt\\
  &=
  \frac{\Gamma(1+\gamma)}{2\pi i}\Biggl(\int_{\Gamma_0} f(t)\, dt 
  + \int_{\Gamma_{\Lambda, \infty}^\limpm} f(t)\, dt
   + \int_{\Gamma_{\eps, \Lambda}^\limpm} f(t)\, dt
   + \int_{\Gamma_{\eps, c}^\limpm} f(t)\, dt\Biggr)\\
  &\quad + \Gamma(1+\gamma)\sum_{\substack{t\in \P(f)\\ |{\Im(t)}|\in (0, \, \Lambda)}} \Res(f, t), 
\end{align}
where $\Gamma_0$, $\eps$ are as before and
\begin{align}
  &\Gamma_{\Lambda, \infty}^\limpm = (c\pm i \Lambda, c\pm i\infty), \\
  &\Gamma_{\eps, \Lambda}^\limpm = (-\eps\pm i0, -\eps \pm i \Lambda), \\
  &\Gamma_{\eps, c}^\limpm = (-\eps\pm i \Lambda, c\pm i \Lambda).
\end{align}

The integral over $\Gamma_0$ can be computed precisely as in the case of rational $\beta$:
\begin{align}
   \frac{\Gamma(1+\gamma)}{2\pi i} \int_{\Gamma_0} f(t)\, dt 
   &=
   \sum_{k=0}^{M+1} \alpha_{k}\lambda^{2-k+\gamma}
       +o(\lambda^{-M+\gamma+\delta}), 
\end{align}
for any $M\in \N, $ $\delta>0$.

There are now only simple poles, the residues at which can be calculated as before:
\begin{align}
   \sum_{\substack{t\in \P(f)\\ |{\Im(t)}|\in (0, \, \Lambda)}} \hspace{-10pt}\Res(f, t)
   &=
   \frac{\beta^{-\gamma/2}}{(2\pi)^{1+\gamma}}
     \sum_{\substack{k\in \N\\ 2\pi k\sqrt{\beta}<\Lambda}}
  \frac{\sin(\pi k (2\lambda \sqrt{\beta}-(1+2\sigma) \beta-2\tau)- \frac{\pi}{2}(1+\gamma))}{k^{1+\gamma}\sin(\pi k \beta)} \\
  &+
   \frac{\beta^{\gamma/2}}{(2\pi)^{1+\gamma}}
   \sum_{\substack{k\in \N\\ 2\pi k/{\sqrt{\beta}}<\Lambda}} 
   \frac{\sin(\pi k (2\lambda/{\sqrt{\beta}}-2\sigma-(1+2\tau)/\beta)- \frac{\pi}{2}(1+\gamma))}{k^{1+\gamma}\sin(\pi k/\beta)}.
\end{align} 

Moreover, 
\begin{align}
  \biggl|\int_{\Gamma_{\Lambda, \infty}^\limpm} f(t)\, dt \biggr| 
  &\leq
  \frac{e^{c(\lambda-\sigma\sqrt{\beta}-\tau/{\sqrt{\beta}})}}{(e^{c\sqrt{\beta}}-1)(e^{c/{\sqrt{\beta}}}-1)}\int_{c\limpm i\Lambda}^{c\limpm i\infty} |t|^{-1-\gamma}\, dt \\
  &\leq
  \frac{e^{c(\lambda-\sigma\sqrt{\beta}-\tau/{\sqrt{\beta}})}}{c^2}\int_{\Lambda}^{\infty} t^{-1-\gamma}\, dt \\
  &= \frac{e^{c(\lambda-\sigma\sqrt{\beta}-\tau/{\sqrt{\beta}})}}{\gamma c^2}\Lambda^{-\gamma}, 
\end{align}
since $e^{x}-1\geq x$, for $x\geq 0$. Furthermore, 
\begin{align}
  \biggl| \int_{\Gamma_{\eps, \Lambda}^\pm} f(t)\, dt\biggr|
  &\leq
  \frac{e^{-\eps(\lambda-\sigma\sqrt{\beta}-\tau/{\sqrt{\beta}})}}{(1-e^{-\eps\sqrt{\beta}})(1-e^{-\eps/{\sqrt{\beta}}})}\int_{0}^{\Lambda}(t^2+\eps^2)^{-(1+\gamma)/2}\, dt\\
  &\leq
  \frac{4e^{-\eps(\lambda-\sigma\sqrt{\beta}-\tau/{\sqrt{\beta}})}}{\eps^2}\eps^{-\gamma} \int_{0}^{\frac{\Lambda}{\sqrt{\Lambda^2+\eps^2}}}(1-z^2)^{-1+\gamma/2}\, dz\\
  &\leq \frac{2\sqrt{\pi}\, \Gamma(\tfrac{\gamma}2)}{\Gamma(\tfrac{1+\gamma}2)}e^{-\eps(\lambda-\sigma\sqrt{\beta}-\tau/{\sqrt{\beta}})}\eps^{-2-\gamma}\\
  &= o(\lambda^{-M+\gamma+\delta}), 
\end{align}
where we used that $1-e^{-x}\geq x/2$, for $x\geq0$, and the change of variables $t=\frac{\eps z}{\sqrt{1-z^2}}$.

Finally, for the two last segments of the contour we firstly have that by changing $\Lambda$ by something smaller than $2\pi\min\{\sqrt{\beta}, 1/{\sqrt{\beta}}\}$ we can choose $\Lambda$ so that $\dist(i\Lambda, \P(f)) \geq \frac{\pi}{2} \min\{\sqrt{\beta}, 1/{\sqrt{\beta}}\}$, that is $\dist(\Lambda, 2\pi\sqrt{\beta}\Z \cup 2\pi/{\sqrt{\beta}}\Z)\geq \frac{\pi}{2} \min\{\sqrt{\beta}, 1/{\sqrt{\beta}}\}$. Hence 
\begin{align}
 &\dist(\Lambda\sqrt{\beta}, 2\pi\Z)\geq \dist(\Lambda\sqrt{\beta}, 2\pi\beta\Z \cup 2\pi\Z)\geq \tfrac{\pi}{2} \sqrt{\beta}\min\{\sqrt{\beta}, \tfrac1{\sqrt{\beta}}\}= \tfrac\pi2 \min\{\beta, 1\}, \\
 &\dist(\tfrac\Lambda{\sqrt{\beta}}, 2\pi\Z)\geq \dist(\tfrac\Lambda{\sqrt{\beta}}, 2\pi\Z \cup \tfrac{2\pi}\beta\Z)\geq \tfrac\pi{2\sqrt{\beta}}\min\{\sqrt{\beta}, \tfrac1{\sqrt{\beta}}\} = \tfrac{\pi}{2} \min\{1, \tfrac1\beta\}.
\end{align}
For $\Re(z)\geq -\log(2)$, 
\begin{equation}
  |e^{z}-1|^2= e^{2\Re(z)}-2e^{\Re(z)}\cos(\Im(z))+1\geq  1- \cos(\Im(z))\geq \frac{2}{\pi^2}\dist(\Im(z), 2\pi \Z)^2.
\end{equation}
Here the first inequality follows from that $g(x, y)=e^{2x}-(2e^x-1)\cos(y)$ is non-negative when $x\geq -\log(2)$. Indeed, if $\cos(y)\leq 0$ this is clearly the case, and if $\cos(y)\geq 0$ this can be seen by writing $g$ as $(e^x-\cos(y))^2+(1-\cos(y))\cos(y)$.

For $t\in\Gamma_{\eps, c}^\limpm$, we thus have that
$|(1-e^{t\sqrt{\beta}})(1-e^{t/{\sqrt{\beta}}})|\geq \frac{1}{2}\min\{\beta, 1\}\min\{1, 1/\beta\}= \frac{1}{2}\min\{\beta, 1/\beta\}.$ Therefore
\begin{align}
  \biggl|\int_{\Gamma_{\eps, c}^\pm} f(t)\, dt\biggr|
  &\leq \frac{2e^{c(\lambda-\sigma\sqrt{\beta}-\tau/{\sqrt{\beta}})}}{ \min\{\beta, 1/\beta\}} \int_{-\eps\limpm i\Lambda}^{c \limpm i \Lambda} |t|^{-1-\gamma}\, dt\\
  &\leq 
  \frac{2e^{c(\lambda-\sigma\sqrt{\beta}-\tau/{\sqrt{\beta}})}}{ \min\{\beta, 1/\beta\}} \Lambda^{-1-\gamma}(c+\eps).
\end{align}

What remains is to choose $\Lambda, c$ appropriately. If $c= O(\lambda^{-\alpha})$ and $\Lambda=O(\lambda^{\beta})$, for some $\alpha\geq 1$ and $\beta>0$, then the errors are of orders
\begin{equation}
  (\lambda^{-\alpha}+\log(\lambda)\lambda^{-1})\lambda^{-\beta(1+\gamma)}\sim\log(\lambda)\lambda^{-1-\beta(1+\gamma)}, \qquad \lambda^{2\alpha-\beta\gamma}, \qquad \lambda^{-M+\gamma+\delta}.
\end{equation}
The errors contributing are thus only the last two. Hence larger $\alpha$ only makes things worse so we choose $\alpha=1$, and $\beta$ so that $2-\beta \gamma = -M+\gamma, $ that is $\beta= \frac{M+2-\gamma}{\gamma}$. This choice yields the desired expansion with the claimed remainder term $o(\lambda^{-M+\gamma+\delta})$, for any $\delta>0$.

\subsection{Bounding \texorpdfstring{$\osc(\beta, \lambda)$}{Osc(beta, lambda)}}\label{sec:Bounding_Osc}
The only remaining part to complete the proof of Theorem~\ref{thm:asymptotics} is to prove that the sum of oscillatory terms is $O(\lambda)$ uniformly for $\beta$ in compact subsets of~$\R_\limplus$.
To this end we make use of the following one-dimensional asymptotic expansion:

\begin{lemma}[{\cite[Lemma~2.1]{MR1062904}}]\label{lem:1Dsum}
  For $\gamma >0$ we have an expansion
  \begin{equation}
    \sum_{k=1}^\infty (\lambda-k)^\gamma_\limplus = 
    \sum_{k=0}^{\lceil \gamma \rceil} \rho_k(\gamma) \lambda^{1+\gamma-k}+O(1), 
  \end{equation}
  as $\lambda \to \infty$. 
\end{lemma}

Using Lemma~\ref{lem:1Dsum} we find that
\begin{align}
  R^\gamma_{\sigma, \tau}(\beta, \lambda)
  &=
  \sum_{k\in \N^2}(\lambda-(k_1+\sigma)\sqrt{\beta}-(k_2+\tau)/{\sqrt{\beta}})_\limplus^\gamma\\
  &=
  \frac{1}{\beta^{\gamma/2}}\sum_{k_1= 1}^{\lf\lambda/{\sqrt{\beta}}-\tau/\beta-\sigma\rf}
  \Bigl(
    \sum_{n=0}^{\lceil 1+\gamma \rceil} \rho_n(\gamma) (\sqrt{\beta}\lambda-\tau-(k_1+\sigma)\beta)^{1+\gamma-n} + O(1)
  \Bigr)\\
  &=
  \sum_{n=0}^{\lceil 1+\gamma \rceil} 
  \Bigl(
  \beta^{1-n+\gamma/2}\rho_n(\gamma) 
    \sum_{k_1\geq 1}(\lambda/{\sqrt{\beta}}-\tau/\beta-\sigma-k_1)_\limplus^{1+\gamma-n}
  \Bigr)+ O(\lambda)\\
  &=
  \sum_{\substack{n, m\geq 0\\m+n< 2+\gamma}}\beta^{(n+m)/2}\rho_n(\gamma) \rho_{m}(1+\gamma-n) \lambda^{2+\gamma-n-m}\Bigl(1-\frac{\sigma\beta+\tau}{\lambda\sqrt{\beta}}\Bigr)_\limplus^{2+\gamma-n-m}\\
  &\quad + O(\lambda), 
\end{align}
where the error is uniform for $\beta$ on compact subsets of $\R_\limplus$. By expanding the $(1-c/\lambda)^\eta$ terms in the sum up to $O(\lambda^{-1-\gamma+n+m})$ we obtain an asymptotic expansion of $R_{\sigma, \tau}^\gamma$ up to~$O(\lambda)$. Comparing this to the precise asymptotics we obtained above leads us to conclude that the $\osc(\beta, \lambda)=O(\lambda)$ locally uniformly in $\beta$.

\bigskip\noindent{\bf Acknowledgements.} The author is grateful to Didier Robert for providing a copy of the paper~\cite{MR1061661}, and to the anonymous referee whose comments helped improve the quality of the paper and significantly simplify the proof of Theorem~\ref{thm:HeatKernel}. The author also wishes to thank Katie Gittins, Ari Laptev, Richard Laugesen and Douglas Lundholm for discussions and helpful suggestions. Financial support from the Swedish Research Council grant no.~2012-3864 is gratefully acknowledged.


\bibliographystyle{amsplain}

\begin{thebibliography}{10}

\bibitem{AizenmanLieb}
M.~Aizenman and E.~H. Lieb, \emph{On semiclassical bounds for eigenvalues of
  {S}chr\"odinger operators}, Phys. Lett. A \textbf{66} (1978), no.~6, 
  427--429.

\bibitem{MR3001382}
P.~R.~S. Antunes and P.~Freitas, \emph{Optimal spectral rectangles and lattice
  ellipses}, Proc. R. Soc. Lond. Ser. A Math. Phys. Eng. Sci. \textbf{469}
  (2013), no.~2150.

\bibitem{AriturkLaugesen}
S.~Ariturk and R.~S. Laugesen, \emph{Optimal stretching for lattice points
  under convex curves}, Portugaliae Mathematica \textbf{74} (2017), no.~2, 
  91--114.

\bibitem{MR3556370}
M.~{van den Berg}, D.~Bucur, and K.~Gittins, \emph{Maximising
  {N}eumann eigenvalues on rectangles}, Bull. Lond. Math. Soc. \textbf{48}
  (2016), no.~5, 877--894.

\bibitem{vdBergGittins}
M.~{van den Berg} and K.~{Gittins}, \emph{{Minimising Dirichlet
  eigenvalues on cuboids of unit measure}}, Mathematika \textbf{63} (2017), 
  468--482.

\bibitem{MR1686426}
R.~de~la Bret\`eche, \emph{Preuve de la conjecture de {L}ieb--{T}hirring dans
  le cas des potentiels quadratiques strictement convexes}, Ann. Inst. H.
  Poincar\'e Phys. Th\'eor. \textbf{70} (1999), no.~4, 369--380.

\bibitem{GittinsLarson}
K.~Gittins and S.~Larson, \emph{{Asymptotic behaviour of cuboids optimising
  Laplacian eigenvalues}}, Integral equations and operator theory \textbf{89}
  (2017), no.~5, 607--629.

\bibitem{GuoWang}
J.~Guo and W.~Wang, \emph{Lattice points in stretched model domains of finite
  type in $\mathbb{R}^d$}, J. Number Theory \textbf{191} (2017), 273--288.

\bibitem{MR1061661}
B.~Helffer and D.~Robert, \emph{Riesz means of bound states and semiclassical
  limit connected with a {L}ieb-{T}hirring's conjecture}, Asymptotic Anal.
  \textbf{3} (1990), no.~2, 91--103.

\bibitem{MR1079775}
B.~Helffer and D.~Robert, \emph{Riesz means of bounded states and semi-classical limit connected
  with a {L}ieb-{T}hirring conjecture. {II}}, Ann. Inst. H. Poincar\'e Phys.
  Th\'eor. \textbf{53} (1990), no.~2, 139--147.

\bibitem{MR1062904}
B.~Helffer and J.~Sj\"ostrand, \emph{On diamagnetism and de {H}aas-van {A}lphen
  effect}, Ann. Inst. H. Poincar\'e Phys. Th\'eor. \textbf{52} (1990), no.~4, 
  303--375.

\bibitem{MR3681143}
A.~Henrot (ed.), \emph{Shape optimization and spectral theory}, de Gruyter
  Open, Warsaw, 2017.

\bibitem{MR2310176}
A.~Ivi\'c, E.~Kr\"atzel, M.~K\"uhleitner, and W.~G. Nowak, \emph{Lattice points
  in large regions and related arithmetic functions: recent developments in a
  very classic topic}, Elementare und analytische {Z}ahlentheorie, Schr. Wiss.
  Ges. Johann Wolfgang Goethe Univ. Frankfurt am Main, vol.~20, Franz Steiner
  Verlag Stuttgart, Stuttgart, 2006, pp.~89--128.

\bibitem{Laptev2}
A.~Laptev, \emph{Dirichlet and {N}eumann eigenvalue problems on domains in
  {E}uclidean spaces}, J. Funct. Anal. \textbf{151} (1997), no.~2, 531--545.

\bibitem{MR1747896}
A.~Laptev, \emph{On the {L}ieb--{T}hirring conjecture for a class of potentials}, 
  The {M}az$'$ya anniversary collection, {V}ol. 2 ({R}ostock, 1998), Oper.
  Theory Adv. Appl., vol. 110, Birkh\"auser, Basel, 1999, pp.~227--234.

\bibitem{LaugesenLiu1}
R.~S. Laugesen and S.~Liu, \emph{Optimal stretching for lattice points and
  eigenvalues}, Arkiv f\"or matematik \textbf{56} (2018), 111--145.

\bibitem{LaugesenLiu2}
R.~S. Laugesen and S.~Liu, \emph{Shifted lattices and asymptotically optimal ellipses}, The Journal
  of Analysis \textbf{26} (2018), 71--102.

\bibitem{Marshall}
N.~F. Marshall, \emph{Stretching convex domains to capture many lattice
  points}, Int. Math. Res. Not. IMRN, published online (2018).

\bibitem{MarshallSteinerberger}
N.~F. Marshall and S.~Steinerberger, \emph{Triangles capturing many lattice
  points}, Mathematika \textbf{64} (2018), 551--582.

\bibitem{MR3309484}
W.~G. Nowak, \emph{Integer points in large bodies}, Topics in mathematical
  analysis and applications, Springer Optim. Appl., vol.~94, Springer, Cham, 
  2014, pp.~583--599.


\bibitem{NIST}
F.~W.~J. Olver, A.~B. {Olde Daalhuis}, D.~W. Lozier, B.~I. Schneider, R.~F. Boisvert, C.~W. Clark,
 B.~R. Miller and B.~V. Saunders (eds.), 
 \emph{NIST Digital Library of Mathematical Functions}, Release 1.0.19.


\end{thebibliography}
\def\myarXiv#1#2{\href{http://arxiv.org/abs/#1}{\texttt{arXiv:#1\,[#2]}}}

\end{document}